\documentclass{article}
\usepackage{amsfonts}
\usepackage{amsmath,amssymb}
\usepackage{graphicx}

\begin{document}
\newcommand{\be}{\begin{equation}}
\newcommand{\ee}{\end{equation}}
\newcommand{\ba}{\begin{array}}
\newcommand{\ea}{\end{array}}
\newcommand{\bea}{\begin{eqnarray*}}
\newcommand{\eea}{\end{eqnarray*}}
\newcommand{\bean}{\begin{eqnarray}}
\newcommand{\eean}{\end{eqnarray}}

\newtheorem{theorem}{Theorem}[section]
\newtheorem{lemma}{Lemma}[section]
\newtheorem{remark}{Remark}[section]
\newtheorem{proposition}{Proposition}[section]
\newtheorem{definition}{Definition}[section]
\newtheorem{corollary}{Corollary}[section]
\newtheorem{algorithm}{Algorithm}[section]

\newcommand{\lc}{\mathrel{\raise2pt\hbox{${\mathop<\limits_{\raise1pt\hbox{\mbox{$\sim$}}}}$}}}
\newcommand{\gc}{\mathrel{\raise2pt\hbox{${\mathop>\limits_{\raise1pt\hbox{\mbox{$\sim$}}}}$}}}
\newcommand{\ec}{\mathrel{\raise1pt\hbox{${\mathop=\limits_{\raise2pt\hbox{\mbox{$\sim$}}}}$}}}
\newcommand{\tol}{\delta}
\newcommand{\supp}{{\rm supp}}
\newcommand{\III}{|\!|\!|}
\newcommand{\Label}[1]{\label{#1}{{\mbox{$\;$\fbox{\tiny\tt #1}$\;$}}}}
\renewcommand{\Label}[1]{\label{#1}}
\renewcommand{\theequation}{\arabic{section}.\arabic{equation}}
\renewcommand{\thetheorem}{\arabic{section}.\arabic{theorem}}
\renewcommand{\thelemma}{\arabic{section}.\arabic{lemma}}
\renewcommand{\theproposition}{\arabic{section}.\arabic{proposition}}
\renewcommand{\thedefinition}{\arabic{section}.\arabic{definition}}
\renewcommand{\thecorollary}{\arabic{section}.\arabic{corollary}}
\renewcommand{\thealgorithm}{\arabic{section}.\arabic{algorithm}}
\renewcommand{\thefigure}{\arabic{section}.\arabic{figure}}

\newcommand{\lan}{\langle}
\newcommand{\curl}{{\bf curl \;}}
\newcommand{\rot}{{\rm curl}}
\newcommand{\grad}{{\bf grad \;}}
\newcommand{\dvg}{{\rm div \,}}
\newcommand{\ran}{\rangle}
\newcommand{\bR}{\mbox{\bf R}}
\newcommand{\bRn}{{\bf R}^3}
\newcommand{\Coinf}{C_0^{\infty}}
\newcommand{\disp}{\displaystyle}
\newcommand{\proof}{\vspace{1ex}\noindent{\em Proof}. \ }
\newcommand{\ra}{\rightarrow}
\newcommand{\Ra}{\Rightarrow}
\newcommand{\ud}{u_{\delta}}
\newcommand{\Ed}{E_{\delta}}
\newcommand{\Hd}{H_{\delta}}
\newcommand\varep{\varepsilon}
\title{Convergence and Optimal Complexity of Adaptive Finite Element
Methods\thanks{This work was supported by the National Natural
Science Foundation of China (10425105 and 10871198) and the National
Basic Research Program of China (2005CB321704).}}

\author{Lianhua He\thanks{LSEC, Institute of Computational Mathematics and
Scientific/Engineering Computing, Academy of Mathematics and Systems
Science, Chinese Academy of Sciences, Beijing 100190, China and,
Graduate University of Chinese Academy of Sciences, Beijing 100190,
China ({helh@lsec.cc.ac.cn}).}  \and Aihui Zhou\thanks{LSEC,
Institute of Computational Mathematics and Scientific/Engineering
Computing, Academy of Mathematics and Systems Science, Chinese
Academy of Sciences, Beijing 100190, China
({azhou@lsec.cc.ac.cn}).}}


\date{}

\maketitle
\begin{abstract}
In this paper, we study adaptive finite element approximations in a
perturbation framework, which makes use of the existing adaptive
finite element analysis of a linear symmetric elliptic problem. We
prove the convergence and complexity of  adaptive finite element
methods for a class of elliptic partial differential equations. For
illustration, we apply the general approach to obtain the
convergence and complexity of adaptive finite element methods for  a
nonsymmetric problem, a nonlinear problem as well as an unbounded
coefficient eigenvalue problem.
\end{abstract}



\section{Introduction}

The purpose of this paper is to study the convergence and complexity
of  adaptive finite element computations for a class of elliptic
partial differential equations of second order and to apply our
general approach to  three  problems: a nonsymmetric problem, a
nonlinear problem, and an eigenvalue problem with an unbounded
coefficient. One technical tool for motivating this work is the
relationship between the general problem and a linear symmetric
elliptic problem, which is derived from some perturbation arguments
(see Theorem \ref{thm-general-boundary} and Lemma
\ref{lemma-bound-general} ).

Since Babu{\v s}ka and Vogelius \cite{babuska-vogelius-84} gave an
analysis of an adaptive finite element method (AFEM) for linear
symmetric elliptic problems in $1D$, there are a number of work on
the convergence and complexity of adaptive finite element methods in
the literature. For instance, D{\" o}rfler \cite{dorfler-96}
presented the first multidimensional convergence result, which has
been improved and generalized in
\cite{binev-dehmen-devore-04,cascon-kreuzer-nochetto-siebert-07,Dai-Xu-Zhou-08,
mekchay-nochetto-05,morin-nochetto-siebert-00,morin-nochetto-siebert-02,morin-siebrt-veeser-08,
stevenson-06a}.
For a nonsymmetric problem, in particular,  Mekchay and Nochetto
\cite{mekchay-nochetto-05} imposed a quasi-orthogonality property
instead of the Pythagoras equality to prove the convergence of AFEM
while  Morin, Siebrt, and Veeser \cite{morin-siebrt-veeser-08}
showed the convergence of error and estimator simultaneously with
the strict error reduction and derived the convergence of the
estimator by exploiting the (discrete) local lower but not the upper
bound. To our best knowledge, however, there has been no any work on
the complexity of AFEM for nonsymmetric elliptic problems in the
literature.  In this paper, we can get the convergence and optimal
complexity of nonsymmetric problems from our general approach. For a
nonlinear problem,
Chen, Holst and Xu  \cite{chen-holst-xu} proved the convergence of
an adaptive finite element algorithm for  Poisson-Boltzmann equation
while we are able to obtain the convergence and optimal complexity
of AFEM for a class of nonlinear problems now. For a smooth
coefficient eigenvalue problem, Dai, Xu, and Zhou
\cite{Dai-Xu-Zhou-08} gave the convergence and  optimal complexity
of AFEM for symmetric elliptic eigenvalue problems with piecewise
smooth coefficients (see, also convergence analysis of a special
case \cite{Garau-Morin-Zuppa-08,Giani-Graham-09}). In this paper we
will derive similar results for unbounded coefficient eigenvalue
problems from our general conclusions, too. We mention that a
similar perturbation approach was used in \cite{Dai-Xu-Zhou-08}.

This paper is organized as follows. In section 2 we review some
existing results on the  convergence and complexity analysis of AFEM
for the typical problem. In section 3 we generalize results to a
general model problem by using a perturbation argument. In section 4
and section 5, we provide three typical applications for
illustration, including theory and numerics.

\section{Adaptive FEM for a typical problem}\label{sc:linear
boundary} In this section, we review some existing results on the
convergence and complexity analysis of AFEM  for a boundary value
problem in the literature.

Let $\Omega\subset \mathbb{R}^d(d\ge 2)$ be a bounded polytopic
domain. We shall use the standard notation for Sobolev spaces
$W^{s,p}(\Omega)$ and their associated norms and seminorms, see,
e.g., \cite{adams-75,ciarlet-lions-91}. For $p=2$, we denote
$H^s(\Omega)=W^{s,2}(\Omega)$ and $H^1_0(\Omega)=\{v\in H^1(\Omega):
 v\mid_{\partial\Omega}=0\}$, where $v\mid_{\partial\Omega}=0$ is understood in
the sense of trace, $\|\cdot\|_{s,\Omega}= \|\cdot\|_{s,2,\Omega}$.
Throughout this
paper, we shall use $C$ to denote a generic positive constant which
may stand for different values at its different occurrences.
We will also use $A\lc B$ to mean that $A\le C B$ for some constant
$C$  that is independent of mesh parameters. All constants involved
are independent of mesh sizes.
\subsection{A boundary value problem}
 Consider a homogeneous boundary value
problem:
\begin{equation}\label{problem}
\left\{\begin{array}{rl}
 Lu :=-\nabla \cdot (\mathbf{A}\nabla u)&=f \,\,\, \mbox{in} \quad \Omega,\\
u &= 0\,\,\,\mbox{on}~~\partial\Omega,
\end{array}\right.
\end{equation}
 where $\mathbf{A}:\Omega\rightarrow  \mathbb{R}^{d\times d}$ is piecewise
Lipschitz over initial triangulation $\mathcal {T}_{0}$, for $x\in
\Omega$  matrix $\mathbf{A(x)}$ is symmetric and positive definite
with smallest eigenvalue uniformly bounded away from 0, and $f\in
L^{2}(\Omega)$.

\begin{remark}\label{remark-boundary}
The choice of homogeneous boundary condition is made for ease of
presentation, since similar results are valid for other boundary
conditions \cite{cascon-kreuzer-nochetto-siebert-07}.
\end{remark}

The weak form of (\ref{problem}) reads as follows: Find $u\in
H^1_0(\Omega)$ such that
\begin{eqnarray} \label{variation}
a(u, v)= (f, v) \qquad\forall v\in H^1_0(\Omega),
\end{eqnarray}
where $a(\cdot,\cdot)=(\mathbf{A}\nabla \cdot,\nabla \cdot)$. It is
seen that $a(\cdot,\cdot)$ is bounded and coercive on
$H^1_0(\Omega)$, i.e., for any $w, v\in H^1(\Omega)$ there exist
constants $0<c_a\le C_a<\infty$ such that
\begin{eqnarray*}
|a(w, v)| \leq C_a \|w\|_{1, \Omega} \|v\|_{1, \Omega} \quad
\textnormal{and} \quad  c_a\|v\|^2_{1,\Omega} \leq a(v,v) ~~\forall
v\in H^1_0(\Omega).
\end{eqnarray*}
The energy norm $\|\cdot\|_{a,\Omega}$ , which is equivalent to
$\|\cdot\|_{1,\Omega}$ , is defined by
$\|w\|_{a,\Omega}=\sqrt{a(w,w)}$ . It is known that
(\ref{variation}) is well-posed, that is, there exists a unique
solution for any $f\in H^{-1}(\Omega)$.

Let $\{\mathcal{T}_h\}$ be a shape regular family of nested
conforming meshes over $\Omega$: there exists a constant
$\gamma^{\ast}$ such that
\begin{eqnarray*}
\frac{h_\tau}{\rho_\tau} \leq \gamma^{\ast} ~~~ \forall \tau \in
\bigcup_h\mathcal{T}_h,
\end{eqnarray*}
where, for each $\tau\in \mathcal{T}_h$, $h_\tau$ is the diameter of
$\tau$,  $\rho_\tau$ is the diameter of the biggest ball contained
in $\tau$, and
 $h=\max\{h_\tau: \tau\in \mathcal{T}_h\}$. Let $\mathcal{E}_h$ denote
the set of interior sides (edges or faces) of $\mathcal{T}_h$. Let
$S_0^h(\Omega) \subset H_0^1(\Omega)$ be a family of nested finite
 element spaces consisting of continuous piecewise polynomials
 over $\mathcal{T}_h$ of fixed degree $n \geq 1$, which vanish on
$\partial\Omega$.

Define the Galerkin-projection $P_h:H^1_0(\Omega) \rightarrow
S_0^h(\Omega)$ by
\begin{eqnarray}\label{projection}
a(u-P_h u, v) =0 \quad\forall v\in S_0^h(\Omega).
\end{eqnarray}
For any $u\in H_0^1(\Omega)$, there apparently hold:
\begin{equation*}\label{stable1}
      \|P_hu\|_{a,\Omega}\lc \|u\|_{a,\Omega}
\quad \textnormal{and} \quad
     \lim_{h\rightarrow 0} \|u-P_hu\|_{a,\Omega}=0.
\end{equation*}

Now we introduce the following quantity:
\begin{eqnarray*}
\rho_{_{\Omega}}(h)=\sup_{ f\in L^2(\Omega),\|f\|_{0,\Omega}=1}
\inf_{v\in S_0^h(\Omega)}\|L^{-1}f-v\|_{a,\Omega},
\end{eqnarray*}
then $\rho_{_{\Omega}}(h)\to 0$ as $h\to 0$ (see, e.g.,
\cite{babuska-osborn-89,xu-zhou-00}).

A standard finite element scheme for (\ref{variation}) is: Find
 $u_h \in S_0^h(\Omega)$ satisfying
 \begin{eqnarray}\label{dis-fem}
  a(u_h, v) =(f, v)   \qquad \forall  v \in S_0^h(\Omega).
\end{eqnarray}
By definition (\ref{projection}), we know that $u_h=P_hu$.

 By a contradiction
argument, we have (c.f., e.g., \cite{zhou-07})
\begin{lemma}\label{lemma1}
As operators over $H^1_0(\Omega)$, there holds
\begin{eqnarray*}
 \label{operator-property}
\lim_{h\rightarrow 0}\|\mathcal {K}(I-P_h)\|=0
\end{eqnarray*}
if $\mathcal {K}$ is a compact operator over $H^1_0(\Omega)$.
\end{lemma}

\subsection{Adaptive algorithm}
Given an initial triangulation $\mathcal{T}_0$, we shall generate a
sequence of nested conforming triangulations $\mathcal{T}_k$ using
the following loop:
\begin{eqnarray*}\label{loop}
{\bf SOLVE}\rightarrow {\bf ESTIMATE}\rightarrow {\bf
MARK}\rightarrow {\bf REFINE}.
\end{eqnarray*}
 More precisely to get $\mathcal{T}_{k+1}$ from
$\mathcal{T}_k$ we first solve the discrete equation to get $u_k$ on
$\mathcal{T}_k$. The error is estimated using $u_k$ and used to mark
a set of elements that are to be refined. Elements are refined in
such a way that the triangulation is still shape regular and
conforming.
We assume that the solutions of finite-dimensional problems can be
solved to any accuracy efficiently.\footnote{By the similar
perturbation argument, indeed, it will be seen that some
approximations to the finite-dimensional problem will be
sufficient.} Examples of such optimal solvers are multigrid method
or multigrid-based preconditioned conjugate gradient method.

Now we review the residual type a posteriori error estimators for
finite element solutions of (\ref{problem}). Let $\mathbb{T}$ denote
the class of all conforming refinements by bisection of
$\mathcal{T}_0$. For $\mathcal{T}_h \in \mathbb{T}$ and any $v\in
S^h_0(\Omega)$ we define the element residual
$\tilde{\mathcal{R}}_{\tau}(v)$ and the jump residual
$\tilde{J}_e(v)$ by
 \begin{eqnarray*}\label{residual}
  \tilde{\mathcal{R}}_{\tau}(v) &:=& f-L v = f +\nabla \cdot (\mathbf{A}\nabla v)\qquad \mbox{in}~ \tau\in
  \mathcal{T}_h,\\
  \tilde{J}_e(v) &:=& -\mathbf{A} \nabla v^{+}\cdot \nu^{+} - \mathbf{A} \nabla v^{-}\cdot
  \nu^{-} := [[\mathbf{A}\nabla v]]_e \cdot \nu_e  ~~~~ \mbox{on}~ e\in \mathcal{E}_h,
\end{eqnarray*}
where $e$ is the  common side of elements $\tau^+$ and $\tau^-$ with
unit outward normals $\nu^+$ and $\nu^-$, respectively, and
$\nu_e=\nu^-$. Let $\omega_e$ be the union of elements which share
the side $e$ and $\omega_\tau$ be the union of elements sharing a
side with $\tau$.


For $\tau\in \mathcal{T}_h$, we define the local error indicator
$\tilde{\eta}_h(v, \tau)$ by
  \begin{eqnarray*}\label{error-indicator}
   \tilde{\eta}^2_h(v, \tau) := h_\tau^2\|\tilde{\mathcal{R}}_{\tau}(v)\|_{0,\tau}^2
   + \sum_{e\in \mathcal{E}_h,e\subset\partial \tau
   } h_e \|\tilde{J}_e(v)\|_{0,e}^2
  \end{eqnarray*}
and the oscillation $\widetilde{osc}_h(v,\tau)$ by
\begin{eqnarray*}\label{local-oscillation}
\widetilde{osc}^2_h(v,\tau) :=
h_\tau^2\|\tilde{\mathcal{R}}_{\tau}(v)-\overline{\tilde{\mathcal{R}}_{\tau}(v)}\|_{0,\tau}^2
   + \sum_{e\in \mathcal{E}_h,e\subset\partial \tau
   } h_e \|\tilde{J}_e(v)-\overline{\tilde{J}_e(v)}\|_{0,e}^2,
\end{eqnarray*}
where $\overline{w}$ is the $L^2$-projection of $w\in L^2(\Omega)$
to polynomials of some degree  on $\tau$ or $e$.


Given a subset $\omega \subset \Omega$, we define the error
 estimator $\tilde{\eta}_h(v, \omega)$ and the oscillation  $\widetilde{osc}_h(v,\omega)$ by
\begin{eqnarray*}
\tilde{\eta}^2_h(v, \omega) := \sum_{\tau\in \mathcal{T}_h, \tau
\subset \omega}
  \tilde{\eta}^2_h(v, \tau) \quad \textnormal{and} \quad
\widetilde{osc}^2_h(v, \omega) := \sum_{\tau\in \mathcal{T}_h, \tau
\subset \omega}
  \widetilde{osc}^2_h(v, \tau).
\end{eqnarray*}

 For $\tau\in \mathcal{T}_h$, we also need
notation
\begin{eqnarray*}
\eta^2_h(\mathbf{A},\tau)
:=h_\tau^2(\|\mbox{div}\mathbf{A}\|^2_{0,\infty,\tau}+h^{-2}_{\tau}\|\mathbf{A}\|^2_{0,\infty,\omega_{\tau}})
\end{eqnarray*}
and
\begin{eqnarray*}
osc^2_h(\mathbf{A},\tau)
:=h_\tau^2(\|\mbox{div}\mathbf{A}-\overline{\mbox{div}\mathbf{A}}\|^2_{0,\infty,\tau}+
h^{-2}_{\tau}\|\mathbf{A}-\bar{\mathbf{A}}\|^2_{0,\infty,\omega_{\tau}}),
\end{eqnarray*}
where  $\overline{v}$ is the best $L^{\infty}$-approximation in the
space of discontinuous polynomials of some degree.

Given a subset $\omega \subset \Omega$ we finally set
\begin{eqnarray*}
\eta_h(\mathbf{A},\omega):=\max_{\tau\in \mathcal{T}_h, \tau \subset
\omega}\eta_h(\mathbf{A},\tau) \quad \textnormal{and} \quad
osc_h(\mathbf{A},\omega):=\max_{\tau\in \mathcal{T}_h, \tau \subset
\omega}osc_h(\mathbf{A},\tau).
\end{eqnarray*}

We now recall the well-known upper and lower bounds for the energy
error in terms of the residual-type estimator (see, e.g.,
\cite{mekchay-nochetto-05,morin-nochetto-siebert-02,verfurth-96}).
\begin{theorem}
\textnormal{(Global a posterior upper and lower bounds)}. Let $u \in
H^1_0(\Omega)$ be the solution of (\ref{variation}) and $u_h \in
S^h_0(\Omega)$ be the solution of (\ref{dis-fem}). Then there exist
constants $\tilde{C}_1$, $\tilde{C}_2$ and $\tilde{C}_3>0$ depending
only on the shape regularity $\gamma^{\ast}$, $C_a $ and $c_a$ such
that
\begin{eqnarray}\label{boundary-upper}
  \|u - u_h \|^2_{a,\Omega} \leq \tilde{C}_1 \tilde{\eta}^2_h(u_h, \mathcal{T}_h)
\end{eqnarray}
and
\begin{eqnarray}\label{boundary-lower}
~~~~~~ \tilde{C}_2 \tilde{\eta}^2_h (u_h, \mathcal{T}_h)\le
\|u-u_h\|_{a,\Omega}^2+
\tilde{C}_3\widetilde{osc}^2_h(u_h,\mathcal{T}_h).
\end{eqnarray}
\end{theorem}


We replace the subscript $h$ by an iteration counter called $k$ and
call the adaptive algorithm without oscillation marking as {\bf
Algorithm $D_0$}, which is defined as follows:

 Choose a parameter $0 < \theta <1:$
\begin{enumerate}
\item Pick  any initial mesh $\mathcal{T}_0$, and let $k=0$.
\item
Solve the system on $\mathcal{T}_0$ for the discrete solution $u_0$.
\item Compute the local indicators ${\tilde \eta}_k$.
\item
Construct $\mathcal{M}_k \subset
 \mathcal{T}_k$ by {\bf Marking Strategy $E_0$ } and parameter
 $\theta$.
 \item Refine $\mathcal{T}_k$ to get a new conforming mesh $\mathcal{T}_{k+1}$ by Procedure {\bf REFINE}.
 \item Solve the system on $\mathcal{T}_{k+1}$ for the discrete
 solution $u_{k+1}$.
\item Let $k=k+1$ and go to Step 3.
\end{enumerate}

The marking strategy, which we call {\bf Marking Strategy $E_0$}, is
crucial for our adaptive methods. Now it can be stated by:

 Given a parameter $0<\theta < 1$ :
\begin{enumerate}
 \item Construct a minimal subset $\mathcal{M}_k$ of
 $\mathcal{T}_k$ by selecting some elements
 in $\mathcal{T}_k$ such that
\begin{eqnarray*}
\tilde{\eta}_k(u_k, \mathcal{M}_k)  \geq \theta
 \tilde{\eta}_k(u_k,\mathcal{T}_k).
 \end{eqnarray*}
 \item Mark all the elements in $\mathcal{M}_k$.
 \end{enumerate}

Due to \cite{cascon-kreuzer-nochetto-siebert-07}, the procedure {\bf
REFINE} here is not required to satisfy the Interior Node Property
of \cite{mekchay-nochetto-05,morin-nochetto-siebert-02}.

Given a fixed number $b\geq 1$, for any $\mathcal{T}_k \in
\mathbb{T}$ and a subset $\mathcal{M}_k\subset \mathcal{T}_k$ of
marked elements,
\begin{eqnarray*}
\mathcal{T}_{k+1}={\bf REFINE}(\mathcal{T}_k,\mathcal{M}_k)
\end{eqnarray*}
outputs a conforming triangulation $\mathcal{T}_{k+1}\in
\mathbb{T}$, where at least all elements of $\mathcal{M}_k$ are
bisected $b$ times. We define $R_{\mathcal{T}_k\rightarrow
\mathcal{T}_{k+1}}:=\mathcal{T}_k\backslash(\mathcal{T}_k\cap
\mathcal{T}_{k+1})$ as the set of refined elements, thus
$\mathcal{M}_k\subset R_{\mathcal{T}_k\rightarrow
\mathcal{T}_{k+1}}$.
\begin{lemma}
\textnormal{(Complexity of Refine)}. Assume that $\mathcal{T}_0$
verifies condition (b) of section 4 in \cite{Stevenson-08}. For
$k\geq 0$ let $\{\mathcal{T}_k\}_{k\geq 0}$ be any sequence of
refinements of $\mathcal{T}_0$ where $\mathcal{T}_{k+1}$ is
generated from $\mathcal{T}_{k}$ by $\mathcal{T}_{k+1}={\bf
REFINE}(\mathcal{T}_{k},\mathcal{M}_k)$ with a subset
$\mathcal{M}_k\subset \mathcal{T}_k$. Then
\begin{eqnarray}\label{complexity}
\#\mathcal{T}_k-\#\mathcal{T}_0 \lc \sum_{j=0}^{k-1}\#\mathcal{M}_j
\quad \forall k\geq 1
\end{eqnarray}
is valid, where the hidden constant depends on $\mathcal{T}_0$ and
b.
\end{lemma}

The convergence  of {\bf Algorithm $D_0$} is shown in
\cite{cascon-kreuzer-nochetto-siebert-07}.
\begin{theorem}\label{convergence-boundary}
 Let $\{u_k\}_{k\in \mathbb{N}_0}$ be a sequence of finite element
 solutions corresponding to a sequence of nested finite element
 spaces $\{S^k_0(\Omega)\}_{k\in \mathbb{N}_0}$ produced by {\bf Algorithm $D_0$}.
Then there exist constants $\tilde{\gamma}>0$ and  $\tilde{\xi}\in
(0,1)$ depending only on the shape regularity of meshes, the data
and the marking parameter $\theta$, such that for any two
consecutive iterates we have
\begin{eqnarray*}\label{convergence-neq}
&&\|u-u_{k+1}\|^2_{a,\Omega} + \tilde{\gamma}
\tilde{\eta}^2_{k+1}(u_{k+1},\mathcal{T}_{k+1}) \nonumber \\
 &\leq & \tilde{\xi}^2\big( \|u-u_k\|^2_{a,\Omega} +
\tilde{\gamma} \tilde{\eta}^2_k(u_k, \mathcal{T}_k)\big).
\end{eqnarray*}
Indeed,  constant ${\tilde\gamma}$ has the following form
\begin{eqnarray}\label{gamma-boundary}
\tilde{\gamma} := \frac{1}{(1 + \delta^{-1})\Lambda_1
\eta^2_0(\mathbf{A},\mathcal{T}_0)},
\end{eqnarray}
where
$\eta^2_0(\mathbf{A},\mathcal{T}_0):=\eta^2_{\mathcal{T}_0}(\mathbf{A},\mathcal{T}_0)$,
 $\Lambda_1 :=(d+1)C_0^2/c_a$ with $C_0$ some positive constant and constant $\delta \in (0,1)$.
\end{theorem}

Following \cite{cascon-kreuzer-nochetto-siebert-07,Dai-Xu-Zhou-08},
we have a link between nonlinear approximation theory and the {\bf
AFEM} through the marking strategy as follows.
\begin{lemma}
\label{complexity-general-optimal-marking}\textnormal{(Optimal
Marking).} Let $u_k \in S^k_0(\Omega)$ and $u_{k+1} \in
S^{k+1}_0(\Omega)$ be finite element solutions of (\ref{variation})
over a conforming mesh $\mathcal{T}_k$ and its refinement
$\mathcal{T}_{k+1}$ with marked element $\mathcal{M}_k$. Suppose
that they satisfy the decrease property
\begin{eqnarray*}
  &&\|u - u_{k+1} \|_{a,\Omega}^2 + \tilde{\gamma_{\ast}}
  \widetilde{osc}^2_{k+1}(u_{k+1},\mathcal{T}_{k+1}) \nonumber\\
  & \leq&   \tilde{\beta_{\ast}}^2 \big(\|u-u_k\|_{a,\Omega}^2
  + \tilde{\gamma_{\ast}}\widetilde{osc}^2_k(u_k,\mathcal{T}_k)\big)
 \end{eqnarray*}
with constants $\tilde{\gamma_{\ast}}>0$  and
$\tilde{\beta_{\ast}}\in (0,\sqrt{\frac{1}{2}})$. Then the set
$\mathcal{R}:=R_{\mathcal{T}_k\rightarrow \mathcal{T}_{k+1}}$
satisfies the following inequality
\begin{eqnarray*}\label{complexity-general-optimal-marking-neq1}
\tilde{\eta}_k(u_k,\mathcal {R})\geq
\hat{\theta}\tilde{\eta}_k(u_k,\mathcal{T}_k)
\end{eqnarray*}
with $\hat{\theta}^2 =
\frac{\tilde{C}_2(1-2\tilde{\beta_{\ast}}^2)}{\tilde{C}_0 (
\tilde{C}_1 + (1 +  2 C \tilde{C}_1) \tilde{\gamma_{\ast}})}$ ,
where $C=\Lambda_1 osc^2_0(\mathbf{A},\mathcal{T}_0)$
\textnormal{and} $\tilde{C}_0 = \max(1,
\frac{\tilde{C}_3}{\tilde{\gamma_{\ast}}})$.
\end{lemma}

\section{A general framework}
Let $u\in H^1_0(\Omega)$ satisfy
\begin{eqnarray}\label{Gvariation}
a(u, v)+(Vu,v)= (\ell u, v) \qquad\forall v\in H^1_0(\Omega),
\end{eqnarray}
where $\ell:H^1_0(\Omega)\rightarrow  L^{2}(\Omega)$ is a bounded
operator and $V:H^1_0(\Omega)\rightarrow  L^{2}(\Omega)$ is a linear
bounded operator.

Let $K: L^2(\Omega)\rightarrow H^1_0(\Omega)$  be the operator
defined by
\begin{eqnarray*}\label{def-operator-k}
 a(Kw, v) = (w, v)~~~~\forall w, v \in L^2(\Omega).
\end{eqnarray*}
Then $K$ is a compact operator and (\ref{Gvariation}) becomes as
\begin{eqnarray*}\label{operator-K-1}
  u+KVu=K\ell u.
\end{eqnarray*}

Let  $u_h\in S^h_0(\Omega)$ be a solution of disctetization
 \begin{eqnarray}\label{Gdis-fem}
  a(u_h, v)+(Vu_h,v) =(\ell_{h}u_h, v)   ~~~~ \forall  v \in
  S^h_0(\Omega),
\end{eqnarray}
where $\ell_{h}:S^h_0(\Omega)\rightarrow  L^{2}(\Omega)$ is some
bounded operator. Note that we may view $\ell_{h}$ as a perturbation
to $\ell$, for which we assume that there exists $\kappa_1(h)\in
(0,1)$ such that
\begin{eqnarray}\label{assume}
\|K(\ell u- \ell_h u_h)\|_{a,\Omega}=\mathcal
 {O}(\kappa_1(h))\|u-u_h\|_{a,\Omega},
\end{eqnarray}
where $\kappa_1(h)\rightarrow 0$ as $h\rightarrow 0$.

Note  that (\ref{Gdis-fem}) can be written as
\begin{eqnarray*}\label{operator-K-2}
u_h+P_h KVu_h=P_hK\ell_h u_h,
\end{eqnarray*}
where $P_h$ is defined by (\ref{projection}). We have for
$w^h=K\ell_{h}u_h-KVu_h$ that
\begin{eqnarray}\label{u-w}
u_h=P_hw^h.
\end{eqnarray}

\begin{theorem}\label{thm-general-boundary}
 There exists $\kappa(h)\in (0,1)$ such that $\kappa(h)\rightarrow 0$ as $h\rightarrow
 0$ and
\begin{eqnarray}\label{general-boundary-neq}
 \|u-u_h\|_{a,\Omega}= \|w^h - P_h w^h\|_{a,\Omega} +\mathcal
 {O}(\kappa(h))\|u-u_h\|_{a,\Omega}.
\end{eqnarray}
\end{theorem}
\begin{proof}
 By definition, we have
\begin{eqnarray*}
 u-w^h=K\ell u - K Vu - (K \ell_h u_h - KVu_h) = K(\ell u-\ell_h u_h) +
 KV(u_h-u).
\end{eqnarray*}

Let $\kappa_2(h)=\|KV(I-P_h)\|.$  Since $KV :
H^1_0(\Omega)\rightarrow H^1_0(\Omega)$ is compact, we get from
Lemma \ref{lemma1} that
 $\kappa_2(h)\rightarrow 0$ as $h\rightarrow 0$. Note that
\begin{eqnarray*}
KV(u_h-u)=KV(I-P_h)(u_h-u),
\end{eqnarray*}
we obtain
 \begin{eqnarray}\label{KV-property}
 \|KV(u_h-u)\|_{a,\Omega}=\mathcal
 {O}(\kappa_2(h))\|u-u_h\|_{a,\Omega}.
\end{eqnarray}

Set $\kappa(h)=\kappa_1(h)+\kappa_2(h)$, we
 have that $\kappa(h)\rightarrow 0$ as $h\rightarrow 0$ and
\begin{eqnarray}\label{thm1-neq1}
\|u-w^h\|_{a,\Omega}\leq \tilde{C}\kappa(h)\|u-u_h\|_{a,\Omega}.
\end{eqnarray}
Since (\ref{u-w}) implies
\begin{eqnarray*}
 u - u_h = w^h-P_h w^h +u-w^h,
 \end{eqnarray*}
we get (\ref{general-boundary-neq}) from (\ref{thm1-neq1}). This
completes the proof.
\end{proof}

Theorem \ref{thm-general-boundary} sets up a relationship between
the error estimates of finite element approximations of the general
problem and the associated typical finite element boundary value
solutions, from which various a posteriori error estimators for the
general problem can be easily obtained since the a posteriori error
estimators for the typical boundary value problem have been
well-constructed. In fact, Theorem \ref{thm-general-boundary}
implies that up to the high order term, the error of the general
problem is equivalent to that of the typical problem with $\ell_h
u_h-Vu_h$ as a source
 term. However, the high order term can not be estimated
easily in the analysis of convergence and optimal complexity of AFEM
for the general problem, for instance, for a nonsymmetric problem, a
nonlinear problem and an unbounded coefficient eigenvalue problem.

\subsection{Adaptive algorithm}
 Following the element residual $\tilde{\mathcal{R}}_{\tau}(u_h)$
and the jump residual $\tilde{J}_e(u_h)$ for  (\ref{dis-fem}),  we
define the element residual $\mathcal{R}_{\tau}(u_h)$ and the jump
residual $J_e(u_h)$ for  (\ref{Gdis-fem})  as follows:
\begin{eqnarray*}\label{Gresidual}
  \mathcal{R}_{\tau}(u_h) &:=& \ell_{h}u_h-Vu_h-L u_h =\ell_{h}u_h-Vu_h +\nabla \cdot (\mathbf{A}\nabla u_h)~~~~ \mbox{in}~ \tau\in
  \mathcal{T}_h,\\
  J_e(u_h) &:=& -\mathbf{A} \nabla u_h^{+}\cdot \nu^{+} - \mathbf{A} \nabla u_h^{-}\cdot
  \nu^{-} := [[\mathbf{A}\nabla u_h]]_e \cdot \nu_e  ~~~~ \mbox{on}~ e\in
  \mathcal{E}_h.
\end{eqnarray*}
For $\tau\in \mathcal{T}_h$, we define the local error indicator
$\eta_h(u_h, \tau)$ by
  \begin{eqnarray*}\label{Gerror-indicator}
   \eta^2_h(u_h, \tau) := h_\tau^2\|\mathcal{R}_{\tau}(u_h)\|_{0,\tau}^2
   + \sum_{e\in \mathcal{E}_h,e\subset\partial \tau
   } h_e \|J_e(u_h)\|_{0,e}^2
  \end{eqnarray*}
and the oscillation $osc_h(u_h,\tau)$ by
\begin{eqnarray*}\label{Glocal-oscillation}
osc^2_h(u_h,\tau) :=
h_\tau^2\|\mathcal{R}_{\tau}(u_h)-\overline{\mathcal{R}_{\tau}(u_h)}\|_{0,\tau}^2
   + \sum_{e\in \mathcal{E}_h,e\subset\partial \tau
   } h_e \|J_e(u_h)-\overline{J_e(u_h)}\|_{0,e}^2,
\end{eqnarray*}
where $e$ , $\nu^+$ and $\nu^-$ are defined as those in section
\ref{sc:linear boundary}.

Given a subset $\omega \subset \Omega$, we define the error
 estimator $\eta_h(u_h, \omega)$ by
 \begin{eqnarray}\label{Gerror-estimator}
  \eta^2_h(u_h, \omega) := \sum_{\tau\in \mathcal{T}_h, \tau \subset \omega}
  \eta^2_h(u_h, \tau)
  \end{eqnarray}
and the oscillation $osc_h(u_h,\omega)$ by
\begin{eqnarray}\label{Goscilliation}
  osc^2_h(u_h, \omega) :=  \sum_{\tau\in \mathcal{T}_h, \tau \subset \omega}
  osc^2_h(u_h, \tau).
\end{eqnarray}

Let $h_0\in (0,1)$ be the  mesh size of the initial mesh
$\mathcal{T}_0$ and define $$\tilde{\kappa}(h_0):=\sup_{h\in
(0,h_0]}\kappa(h).$$ Obviously, $\tilde{\kappa}(h_0) \ll 1$ if
$h_0\ll 1$.

To analyze the convergence and complexity of finite element
approximations, we need to establish some relationship between the
two level approximations. We use $\mathcal{T}_H$ to denote a coarse
mesh and $\mathcal{T}_h$ to denote a refined mesh of
$\mathcal{T}_H$. Recall that $w^h=K(\ell_h u_h-Vu_h)$ and
$w^H=K(\ell_H u_H-Vu_H)$.
\begin{lemma}\label{lemma-bound-general}
Let $h, H \in (0, h_0]$, then
 \begin{eqnarray}\label{lemma-bound-general-conc-1}
 \|u - u_h \|_{a, \Omega}= \|w^H - P_h w^H \|_{a, \Omega}
  + \mathcal{O} (\tilde{\kappa}(h_0))\left( \|u - u_h
 \|_{a, \Omega}+  \|u - u_H\|_{a, \Omega}\right),~
\end{eqnarray}
\begin{eqnarray}\label{lemma-bound-general-conc-3}
\eta_{h}(u_h, \mathcal{T}_h) = \tilde{\eta}_{h}(P_h w^H,
\mathcal{T}_h) + \mathcal{O}
 (\tilde{\kappa}(h_0))\left( \|u - u_h\|_{a, \Omega} + \|u -
u_H\|_{a, \Omega}\right),
\end{eqnarray}
and
\begin{eqnarray}\label{lemma-bound-general-conc-2}
osc_h(u_h,\mathcal{T}_h)=\widetilde{osc}_h(P_h w^H, \mathcal{T}_h) +
\mathcal{O} (\tilde{\kappa}(h_0))\left( \|u - u_h\|_{a, \Omega} +
\|u - u_H\|_{a, \Omega}\right).
\end{eqnarray}
\end{lemma}

\begin{proof}
First, we prove (\ref{lemma-bound-general-conc-1}).  It follows that
 \begin{eqnarray*} \|P_h (w^h - w^H)+ u - w^H
\|_{a,\Omega} &\lc& \| w^h - w^H\|_{a,\Omega} + \|u -
w^H\|_{a,\Omega}\\ &\lc& \|u - w^H\|_{a,\Omega} + \|u -
w^h\|_{a,\Omega},
\end{eqnarray*}
which together with (\ref{thm1-neq1}) implies
\begin{eqnarray*}
\|P_h (w^h - w^H) + w^H - u\|_{a,\Omega}\lc \kappa(H) \|u -
u_H\|_{a,\Omega} +  \kappa(h)\|u - u_h\|_{a,\Omega}.
\end{eqnarray*}
Namely,
\begin{eqnarray}\label{lem-bound-general1}
\|P_h (w^h - w^H) + w^H - u\|_{a,\Omega}\lc \tilde{\kappa}(h_0)(\|u
- u_H\|_{a,\Omega} + \|u - u_h\|_{a,\Omega}).
\end{eqnarray}
Observing that identity  (\ref{u-w}) leads to
 \begin{eqnarray*}\label{lem-bound-general2}
u-u_h=w^H - P_h w^H+P_h(w^H -w^h)+u-w^H,
 \end{eqnarray*}
 we  then obtain (\ref{lemma-bound-general-conc-1}) from
 (\ref{lem-bound-general1}).

Next, we turn to  prove (\ref{lemma-bound-general-conc-2}). Due to
$Lw^h=\ell_hu_h-Vu_h$ and $Lw^H=\ell_Hu_H-Vu_H$,  we know that $w^h
- w^H$ is the solution of typical boundary value problem with
$\ell_hu_h-\ell_Hu_H+Vu_H-Vu_h$ as a source term. Since
\begin{eqnarray*}
\tilde{\mathcal{R}}_{\tau}(P_h(w^h-w^H))=\ell_hu_h-\ell_Hu_H+Vu_H-Vu_h-L(P_h(w^h-w^H)),
\end{eqnarray*}
 we have
\begin{eqnarray}\label{temp4}
& &\widetilde{osc}^2_h(P_h(w^h-w^H),\mathcal{T}_h) =\sum_{\tau\in
\mathcal{T}_h}\widetilde{osc}^2_h(E,\tau)\nonumber\\
 &=&\sum_{\tau\in \mathcal{T}_h}\big(h^2_{\tau}\|\tilde{\mathcal{R}}_{\tau}(E)-\overline{\tilde{\mathcal{R}}_{\tau}(E)}\|^2_{0,\tau}
+\sum_{e\in \mathcal{E}_h,e\subset\partial \tau
   } h_e
   \|\tilde{J}_e(E)-\overline{\tilde{J}_e(E)}\|^2_{0,e}\big)\nonumber\\
   &\leq&\sum_{\tau\in
   \mathcal{T}_h}h^2_{\tau}\|\tilde{\mathcal{R}}_{\tau}(E)+LE-\overline{(\tilde{\mathcal{R}}_{\tau}(E)+LE)}\|^2_{0,\tau}\nonumber\\
   && +\sum_{\tau\in\mathcal{T}_h}\big(h^2_{\tau}\|LE-\overline{LE}\|^2_{0,\tau}+
   \sum_{e\in
\mathcal{E}_h,e\subset\partial \tau   } h_e
\|\tilde{J}_e(E)-\overline{\tilde{J}_e(E)}\|^2_{0,e}\big),
\end{eqnarray}
where $E=P_h(w^h-w^H)$.
Following the  proof of Proposition 3.3 in
\cite{cascon-kreuzer-nochetto-siebert-07}, we see that
$$\sum_{\tau\in\mathcal{T}_h}\big(h^2_{\tau}\|LE-\overline{LE}\|^2_{0,\tau}+
   \sum_{e\in
\mathcal{E}_h,e\subset\partial \tau   } h_e
\|\tilde{J}_e(E)-\overline{\tilde{J}_e(E)}\|^2_{0,e}\big)$$ can be
bounded by
\begin{eqnarray*}
\sum_{\tau\in\mathcal{T}_h}C_0^2
osc^2_h(\mathbf{A},\tau)\|P_h(w^h-w^H)\|^2_{1,\omega_{\tau}}\lc
osc^2_h(\mathbf{A},\mathcal{T}_h)\|P_h(w^h-w^H)\|^2_{a,\Omega}.
\end{eqnarray*}Hence using the fact
$osc_h(\mathbf{A},\mathcal{T}_h)\leq
osc_0(\mathbf{A},\mathcal{T}_0)$, we obtain
\begin{eqnarray}\label{temp_osc}
&&\sum_{\tau\in\mathcal{T}_h}\big(h^2_{\tau}\|LE-\overline{LE}\|^2_{0,\tau}+
   \sum_{e\in
\mathcal{E}_h,e\subset\partial \tau   } h_e
\|\tilde{J}_e(E)-\overline{\tilde{J}_e(E)}\|^2_{0,e}\big)\nonumber\\
&\lc&
osc^2_0(\mathbf{A},\mathcal{T}_0)\|P_h(w^h-w^H)\|^2_{a,\Omega}.
\end{eqnarray}

Using the inverse inequality, the bounded property of $V$ and
(\ref{assume}), we get
\begin{eqnarray}\label{rightterm}
& &\big(\sum_{\tau\in
   \mathcal{T}_h}h^2_{\tau}\|\tilde{\mathcal{R}}_{\tau}(E)+LE-
   \overline{(\tilde{\mathcal{R}}_{\tau}(E)+LE)}\|^2_{0,\tau}\big)^{1/2}\nonumber\\
&\lc&\big(\sum_{\tau\in
   \mathcal{T}_h}\|h_{\tau}(\ell_hu_h-\ell_Hu_H+Vu_H-Vu_h)\|^2_{0,\tau}\big)^{1/2}\nonumber\\
&\lc& \|K(\ell_hu_h-\ell_Hu_H)\|_{a,\Omega}+h\|u_H-u_h\|_{a,\Omega}\nonumber\\
&\lc&\|K(\ell_hu_h-\ell u)\|_{a,\Omega}+\|K(\ell_Hu_H-\ell
u)\|_{a,\Omega}\nonumber\\
& & +h\|u-u_H\|_{a,\Omega}+h\|u-u_h\|_{a,\Omega}\nonumber\\
&\lc&\tilde{\kappa}(h_0)\left( \|u - u_h\|_{a, \Omega} + \|u -
u_H\|_{a, \Omega}\right).
\end{eqnarray}
Note that
\begin{eqnarray*}
& &\|P_h (w^h - w^H)\|_{a,\Omega} \lc \| w^h - w^H\|_{a,\Omega} \nonumber\\
&\lc& \|u-w^h\|_{a,\Omega}+\|u-w^H\|_{a,\Omega},
\end{eqnarray*}
which together with (\ref{thm1-neq1}) implies
\begin{eqnarray}\label{temp1}
\|P_h (w^h - w^H)\|_{a,\Omega} &\lc&\tilde{\kappa}(h_0)\left( \|u -
u_h\|_{a, \Omega} + \|u - u_H\|_{a, \Omega}\right).
\end{eqnarray}
Combing (\ref{temp4}), (\ref{temp_osc}), (\ref{rightterm}) and
(\ref{temp1}), we conclude that
\begin{eqnarray}\label{temp5}
\widetilde{osc}_h(P_h(w^h-w^H),\mathcal{T}_h)&\lc&
\tilde{\kappa}(h_0)\left(
\|u - u_h\|_{a, \Omega} + \|u - u_H\|_{a, \Omega}\right).
\end{eqnarray}
Due to $u_h=P_h w^H+P_h(w^h-w^H)$,
 we obtain from the definition of oscillation that
\begin{eqnarray}\label{temp6}
\widetilde{osc}_h(P_h w^h,\mathcal{T}_h)\leq\widetilde{osc}_h(P_h
w^H,\mathcal{T}_h) + \widetilde{osc}_h(P_h (w^h-w^H),\mathcal{T}_h).
\end{eqnarray}
Hence from
$\widetilde{osc}_h(u_h,\mathcal{T}_h)=osc_h(u_h,\mathcal{T}_h)$,
(\ref{temp5}) and (\ref{temp6}), we arrive at
(\ref{lemma-bound-general-conc-2}).

Finally, we  prove (\ref{lemma-bound-general-conc-3}). By
(\ref{boundary-lower}) and (\ref{temp5}), we have
\begin{eqnarray}\label{lem-bound-general5}
& &\tilde{\eta}_h(P_h( w^h - w^H), \mathcal{T}_h)\nonumber\\ &\lc&
\| (w^h - w^H) -P_h(w^h - w^H)\|_{a, \Omega}+
\widetilde{osc}_h(P_h(w^h - w^H), \mathcal{T}_h)
\nonumber\\
&\lc & \|u - w^h\|_{a, \Omega} +  \|u - w^H\|_{a, \Omega} +
\tilde{\kappa}(h_0)\left( \|u - u_h\|_{a, \Omega} + \|u - u_H\|_{a,
\Omega}\right)\nonumber\\
&\lc&\tilde{\kappa}(h_0)\left( \|u - u_h\|_{a, \Omega} + \|u -
u_H\|_{a, \Omega}\right).
\end{eqnarray}
From (\ref{lem-bound-general5}) and the fact that
\begin{eqnarray*}
\tilde{\eta}_h(P_h w^h, \mathcal{T}_h) = \tilde{\eta}_h(P_h w^H +
P_h( w^h - w^H), \mathcal{T}_h),
\end{eqnarray*}
we obtain
\begin{eqnarray*}
\tilde{\eta}_h(P_h w^h, \mathcal{T}_h) = \tilde{\eta}_h(P_h w^H,
\mathcal{T}_h)
 + \mathcal{O} (\tilde{\kappa}(h_0))\left( \|u - u_h\|_{a, \Omega} + \|u -
u_H\|_{a, \Omega}\right),
\end{eqnarray*}
which is nothing but (\ref{lemma-bound-general-conc-3}) since
$\tilde{\eta}_h(P_h w^h, \mathcal{T}_h) =
\eta_h(u_h,\mathcal{T}_h)$.
\end{proof}

\begin{theorem}
Let $h_0 \ll 1$ and $h \in (0, h_0].$ There exist constants $C_1,
C_2$ and  $C_3$, which only depend on the shape regularity constant
$\gamma^{\ast}$, $C_a$ and $c_a$ such that
 \begin{eqnarray}\label{upper-bound}
  \|u-u_h\|^2_{a,\Omega} \leq C_1 \eta^2_h(u_h, \mathcal{T}_h)
 \end{eqnarray}
 and
 \begin{eqnarray}\label{lower-bound}
~~~~~~C_2 \eta^2_h(u_h, \mathcal{T}_h) \le
 \|u-u_h\|_{a,\Omega}^2+ C_3osc_h^2(u_h, \mathcal {T}_h).
 \end{eqnarray}
\end{theorem}
\begin{proof} Recall that $L w^h = \ell_{h}u_h-Vu_h$.
From (\ref{boundary-upper}) and
 (\ref{boundary-lower}) we have
\begin{eqnarray}\label{auxiliary-boundary-problem-upper}
  \|w^h - P_h w^h \|^2_{a,\Omega} \leq \tilde{C}_1 \tilde{\eta}^2_h (P_h w^h, \mathcal{T}_h)
\end{eqnarray}
and
\begin{eqnarray}\label{auxiliary-boundary-problem-lower}
\tilde{C}_2 \tilde{\eta}^2_h (P_h w^h, \mathcal{T}_h) &\le&
\|w^h-P_h
 w^h\|_{a,\Omega}^2+\tilde{C}_3\widetilde{osc}^2_h(P_h w^h,
\mathcal{T}_h) .
\end{eqnarray}
Thus we obtain (\ref{upper-bound}) and (\ref{lower-bound}) from
(\ref{u-w}), (\ref{general-boundary-neq}),
(\ref{auxiliary-boundary-problem-upper}) and
(\ref{auxiliary-boundary-problem-lower}). In particular, we may
choose $C_1$, $C_2$ and $C_3$ satisfying
\begin{eqnarray}\label{coef-eigen-bound}
  C_1 = \tilde{C}_1 (1+ \tilde{C} \tilde{\kappa}(h_0))^2, ~~C_2 = \tilde{C}_2 (1 - \tilde{C}
  \tilde{\kappa}(h_0))^2, ~~C_3 = \tilde{C}_3 (1 - \tilde{C}\tilde{\kappa}(h_0))^2.
\end{eqnarray}
\end{proof}

\begin{remark}
The requirement $h_0\ll 1$ is somehow reasonable for finite element
approximations of (\ref{Gvariation}). We can refer to
\cite{mekchay-nochetto-05} for the initial mesh size requirement in
adaptive finite element computations for nonsymmetirc boundary value
problems.
\end{remark}\vskip 0.2cm

Now we address  step {\bf MARK} of solving (\ref{Gdis-fem}) in
detail, which we call {\bf Marking Strategy $E$}. Similar to {\bf
Marking Strategy $E_0$} for (\ref{dis-fem}), we define {\bf Marking
Strategy $E$} for (\ref{Gdis-fem}) to enforce error reduction as
follows:

 Given a parameter $0<\theta < 1$:
\begin{enumerate}
 \item Construct a minimal subset $\mathcal{M}_k$ of $\mathcal{T}_k$ by selecting some elements in
 $\mathcal{T}_k$ such that
\begin{eqnarray*}
\eta_k(u_k, \mathcal{M}_k)  \geq \theta
 \eta_k(u_k, \mathcal{T}_k).
 \end{eqnarray*}
 \item Mark all the elements in $\mathcal{M}_k$.
 \end{enumerate}

The adaptive algorithm of solving (\ref{Gdis-fem}), which we call
{\bf Algorithm $D$}, is nothing but {\bf Algorithm $D_0$} when {\bf
Marking Strategy $E_0$} is replaced by {\bf Marking Strategy $E$}.
\subsection{Convergence}\label{section-convergence}
 We now prove that  {\bf Algorithm $D$} of
(\ref{Gdis-fem}) is a contraction with respect to the sum of the
energy error plus the scaled error estimator.
\begin{theorem}\label{error-reduction}
Let $\theta \in (0,1)$ and $\{u_k\}_{k\in\mathbb{N}_0}$ be a
sequence of finite element
 solutions  corresponding to a sequence of nested finite element
 spaces $\{S^k_0(\Omega)\}_{k\in \mathbb{N}_0}$ produced by {\bf Algorithm
 $D$}.
 Then there exist constants
$\gamma>0$ and $\xi \in (0,1)$
  depending only on the shape regularity constant $\gamma^{\ast}$, $ C_a $, $c_a$ and the marking parameter
  $\theta$ such that
\begin{eqnarray}\label{error-reduction-neq1}
 & & \|u-u_{k+1}\|_{a,\Omega}^2+\gamma \eta^2_{k+1}(u_{k+1}, \mathcal{T}_{k+1})\nonumber\\
   &\leq& \xi^2 \big(\|u-u_k\|_{a,\Omega}^2
   +\gamma \eta^2_k(u_k, \mathcal{T}_k)\big).
\end{eqnarray}
Here,
\begin{eqnarray}\label{gamma}
\gamma :=  \frac{\tilde{\gamma}}{1 - C_4
\delta_1^{-1}\tilde{\kappa}^2(h_0)}
\end{eqnarray}
with $C_4$ a positive constant, provided $h_0\ll 1$.
\end{theorem}
\begin{proof}
For convenience, we use $u_h$, $u_H$ to denote $u_{k+1}$ and
$u_{k}$, respectively. Thus we only need to prove that for $u_h$ and
$u_H$, there holds,
\begin{eqnarray*}\label{error-reduction-neq-2}
~ \|u-u_h\|_{a,\Omega}^2 + \gamma \eta^2_{h}(u_{h},
\mathcal{T}_h)\leq \xi^2 \big(\|u-u_H\|_{a,\Omega}^2 +\gamma
\eta^2_{H}(u_H, \mathcal{T}_H)\big).
\end{eqnarray*}
We conclude from Theorem \ref{convergence-boundary}, $w^h=K(\ell_h
u_h-Vu_h)$ and $ w^H=K(\ell_H u_H-Vu_H)$ that there exist constants
$\tilde{\gamma}>0$ and $\tilde{\xi}\in (0,1)$ satisfying
\begin{eqnarray*}
& & \|w^H-P_h w^H\|_{a,\Omega}^2 +\tilde{\gamma}\tilde{\eta}^2_h(P_h
w^H, \mathcal{T}_h) \nonumber\\
 &\leq&  \tilde{\xi}^2 \big( \|w^H- P_H w^H \|_{a,\Omega}^2+ \tilde{\gamma}
 \tilde{\eta}^2_H(P_H w^H,\mathcal{T}_H)\big).
 \end{eqnarray*}
Hence use the fact that $u_H=P_Hw^H$, we obtain
\begin{eqnarray}\label{error-reduction-neq-4}
 && \|w^H-P_h
w^H\|_{a,\Omega}^2 +\tilde{\gamma}\tilde{\eta}^2_h(P_h w^H,
\mathcal{T}_h) \nonumber\\&\leq& \tilde{\xi}^2 \big( \|w^H- u_H
\|_{a,\Omega}^2+ \tilde{\gamma} \eta^2_H(u_H, \mathcal{T}_H)\big).
\end{eqnarray}

By (\ref{lemma-bound-general-conc-1}) and
(\ref{lemma-bound-general-conc-3}), there exists a constant
$\hat{C}>0$ such that
\begin{eqnarray*}
 & &\|u - u_h\|_{a, \Omega}^2 + \tilde{\gamma}\eta^2_h(u_h,
 \mathcal{T}_h)\nonumber\\
 &\leq& (1+\delta_1) \|w^H - P_h w^H\|_{a, \Omega}^2+ (1 + \delta_1) \tilde{\gamma} \tilde{\eta}^2_h(P_h w^H,
\mathcal{T}_h) \nonumber\\
 &&+ \hat{C} (1+\delta_1^{-1})\tilde{\kappa}^2(h_0) (\|u - u_h\|_{a, \Omega}^2 + \|u -
 u_H\|_{a, \Omega}^2)\nonumber\\
 &&+\hat{C} (1+\delta_1^{-1}) \tilde{\kappa}^2(h_0)
\tilde{\gamma}(\|u-u_h\|_{a, \Omega}^2+\|u - u_H\|_{a, \Omega}^2),
\end{eqnarray*}
where the Young's inequality is used and $\delta_1\in (0,1)$
satisfies
\begin{eqnarray}\label{error-reduction-neq-delta}
 (1 + \delta_1) \tilde{\xi}^2 <1.
 \end{eqnarray}
It thus follows from (\ref{error-reduction-neq-4}),
(\ref{thm1-neq1}), and identity
 $\tilde{\eta}_H(P_H w^H, \mathcal{T}_H) = \eta_H(u_H,
\mathcal{T}_H)$ that there exists a positive constant $C^{\ast}$
depending on $\hat{C}$ and $\tilde{\gamma}$ such that
\begin{eqnarray*}
& & \|u - u_h\|_{a, \Omega}^2 + \tilde{\gamma}\eta^2_h(u_h,
 \mathcal{T}_h)\nonumber\\
&\le& (1+ \delta_1) \tilde{\xi}^2 \big(\|w^H- u_H \|_{a, \Omega}^2+
\tilde{\gamma} \eta^2_H(u_H, \mathcal{T}_H)\big)\nonumber\\
& & + C^{\ast} \delta_1^{-1} \tilde{\kappa}^2(h_0) (\|u-u_h\|_{a,
\Omega}^2+\|u - u_H\|_{a, \Omega}^2)\nonumber\\
 &\leq& (1+ \delta_1) \tilde{\xi}^2 \left(\big( 1 +
\tilde{C}\tilde{\kappa}(h_0)\big)^2 \|u -
u_H\|_{a, \Omega}^2 +  \tilde{\gamma} \eta^2_H(u_H, \mathcal{T}_H)\right) \nonumber\\
&&  + C^{\ast} \delta_1^{-1}
\tilde{\kappa}^2(h_0)\left(\|u-u_h\|_{a, \Omega}^2+\|u - u_H\|_{a,
\Omega}^2\right).
\end{eqnarray*}
Hence, if $h_0\ll 1$, then there exists a positive constant $C_4$
depending on $C^{\ast}$ and $\tilde{C}$  such that
\begin{eqnarray*}
&&\|u - u_h\|_{a, \Omega}^2 + \tilde{\gamma}\eta^2_h(u_h,
 \mathcal{T}_h)\nonumber\\ &\leq& (1+
\delta_1) \tilde{\xi}^2\left(\|u - u_H\|^2_{a, \Omega} +
\tilde{\gamma}
\eta^2_H(u_H, \mathcal{T}_H)\right)\nonumber\\
&& +   C_4  \tilde{\kappa}(h_0)  \|u - u_H\|^2_{a, \Omega} + C_4
\delta_1^{-1} \tilde{\kappa}^2(h_0) \|u - u_h\|_{a, \Omega}^2.
\end{eqnarray*}
Consequently,
\begin{eqnarray*}
&&\big(1 - C_4 \delta_1^{-1}  \tilde{\kappa}^2(h_0)\big) \|u -
u_h\|_{a,
\Omega}^2 + \tilde{\gamma}\eta^2_h(u_h, \mathcal{T}_h)\nonumber\\
 & \leq &\big((1+ \delta_1) \tilde{\xi}^2 +  C_4
 \tilde{\kappa}(h_0)\big)\|u - u_H\|_{a, \Omega}^2 + (1+ \delta_1) \tilde{\xi}^2 \tilde{\gamma}
\eta^2_H(u_H, \mathcal{T}_H),
\end{eqnarray*}
that is
\begin{eqnarray*}
& &\|u -u_h\|_{a, \Omega}^2 + \frac{\tilde{\gamma}}{1 - C_4
\delta_1^{-1} \tilde{\kappa}^2(h_0)}\eta^2_h(u_h, \mathcal{T}_h)\nonumber\\
~~~~ &\leq& \frac{ (1+ \delta_1) \tilde{\xi}^2 +  C_4
 \tilde{\kappa}(h_0) }{1 - C_4 \delta_1^{-1}
 \tilde{\kappa}^2(h_0)}\|u -u_H\|_{a, \Omega}^2
  +\frac{(1+ \delta_1) \tilde{\xi}^2 \tilde{\gamma}}{1 - C_4 \delta_1^{-1}
\tilde{\kappa}^2(h_0)}\eta^2_H(u_H, \mathcal{T}_H).
\end{eqnarray*}

Since $h_0\ll 1$ implies ${\tilde r}(h_0)\ll 1$, we have that the
constant $\xi$ defined by
\begin{eqnarray*}
\xi := \left(\frac{ (1+ \delta_1) \tilde{\xi}^2 +  C_4
 \tilde{\kappa}(h_0)}{1 - C_4 \delta_1^{-1}
 \tilde{\kappa}^2(h_0)}\right)^{1/2}
\end{eqnarray*}
satisfying $\xi \in (0,1)$ if $h_0\ll 1$. Therefore,
\begin{eqnarray*}
&&\|u -u_h\|_a^2 +  \frac{\tilde{\gamma}}{1 - C_4
\delta_1^{-1}\tilde{\kappa}^2(h_0)}\eta^2_h(u_h, \mathcal{T}_h)\nonumber\\
 & \leq& \xi^2\left(\|u -u_H\|_{a, \Omega}^2 +
  \frac{(1+ \delta_1) \tilde{\xi}^2 \tilde{\gamma}}{ (1+ \delta_1) \tilde{\xi}^2 +
  C_4
 \tilde{\kappa}(h_0) } \eta^2_H(u_H, \mathcal{T}_H)\right).
\end{eqnarray*}
Finally, we arrive at (\ref{error-reduction-neq1}) by using the fact
that
$$ \frac{(1+
\delta_1) \tilde{\xi}^2 \tilde{\gamma}}{ (1+ \delta_1) \tilde{\xi}^2
+ C_4
 \tilde{\kappa}(h_0) }<\gamma.$$
  This
completes the proof.
\end{proof}
\subsection{Complexity}\label{optimal-complexity}
We shall study the complexity in a class of functions defined by
\begin{eqnarray*}
\mathcal{A}_{\gamma}^s:=\{v \in H^1_0(\Omega): |v|_{s,\gamma} <
\infty \},
\end{eqnarray*}
where $\gamma>0$ is some constant,
\begin{eqnarray*}
|v|_{s,\gamma} = \sup_{\varepsilon
>0}\varepsilon \inf_{\{\mathcal{T}_k\subset \mathcal{T}_0:
\inf (\|v-v_k\|_{a,\Omega}^2 + (\gamma +1) osc^2_k(v_k,
\mathcal{T}_k))^{1/2} \leq \varepsilon\}} \big(\#\mathcal{T}_k - \#
\mathcal{T}_0\big)^s
\end{eqnarray*}
and  $\mathcal{T}_k\subset \mathcal{T}_0$ means $\mathcal{T}_k$ is a
refinement of $\mathcal{T}_0$. It is seen from the definition that,
for all $\gamma>0$, $\mathcal{A}_{\gamma}^s = \mathcal{A}_{1}^s$.
For simplicity, here and hereafter, we use $\mathcal{A}^s$ to stand
for  $\mathcal{A}_{1}^s$, and use $|v|_{s}$ to denote
 $|v|_{s, \gamma}$. So $\mathcal{A}^s$ is the class of
functions that can be approximated within a given tolerance
$\varepsilon$ by continuous piecewise polynomial functions over a
partition $\mathcal{T}_k$ with number of degrees of freedom
$\#\mathcal{T}_k-\# \mathcal{T}_0 \lc \varepsilon^{-1/s}
|v|_{s}^{1/s}$.

In order to give the proof of the  complexity of {\bf Algorithm D}
for solving (\ref{Gdis-fem}), we need some preparations. Recall that
associated with $u_k$, the solution of (\ref{Gdis-fem}) in each mesh
$\mathcal{T}_k$, $w^k=K(\ell_k u_k-Vu_k)$ satisfies
\begin{eqnarray}\label{auxiliary-boundary-eq}
 a(w^k, v) = (\ell_ku_k-Vu_k, v) ~~~~~ \forall v\in
 H^1_0(\Omega).
\end{eqnarray}

Using the similar procedure as in the proof of Theorem
\ref{error-reduction}, we have
\begin{lemma}\label{complexity-general-boundary}
Let $u_k$ and $u_{k+1}$ be  discrete
 solutions of (\ref{Gdis-fem}) over a
conforming mesh $\mathcal{T}_k$ and its refinement
$\mathcal{T}_{k+1}$ with marked
 set $\mathcal{M}_k$.
Suppose that they
 satisfy the following property
 \begin{eqnarray*}
  &&\|u - u_{k+1} \|_{a,\Omega}^2 + \gamma_{\ast} osc^2_{k+1}(u_{k+1},\mathcal{T}_{k+1}) \nonumber\\
  & \leq&   \beta_{\ast}^2 \big(\|u-u_k\|_{a,\Omega}^2 + \gamma_{\ast} osc^2_k(u_k,\mathcal{T}_k)\big),
 \end{eqnarray*}
 where $ \gamma_{\ast}$ and $\beta_{\ast}$ are some positive constants.
Then for  problem (\ref{auxiliary-boundary-eq}), we have
\begin{eqnarray*}\label{lemma-complexity-general-bound-conc2}
&&  \|w^k - P_{k+1} w^k \|_{a,\Omega}^2 + \tilde{\gamma_{\ast}}
\widetilde{osc}_{k+1}^2(P_{k+1} w^k,\mathcal{T}_{k+1})\nonumber\\
   &\leq& \tilde{\beta_{\ast}}^2 \big( \|w^k -P_k w^k\|^2_{a,\Omega}
+ \tilde{\gamma_{\ast}} \widetilde{osc}^2_k(P_k
w^k,\mathcal{T}_k)\big)
\end{eqnarray*}
with
\begin{eqnarray}\label{complexity-general-boundary-beta-gamma}
\tilde{\beta_{\ast}} := \left(\frac{ (1+ \delta_1) \beta_{\ast}^2 +
C_5
 \tilde{\kappa}(h_0) }{1 - C_5 \delta_1^{-1}
 \tilde{\kappa}^2(h_0)}\right)^{1/2}, \quad \tilde{\gamma_{\ast}} :=
 \frac{\gamma_{\ast}}{1 - C_5 \delta_1^{-1}
\tilde{\kappa}^2(h_0)},
\end{eqnarray}
where $C_5 $ is some positive constant and $\delta_1 \in (0, 1)$ is
some constant as  in the proof of Theorem \ref{error-reduction}.
\end{lemma}
\begin{corollary}\label{optimal-marking}
Let $u_k$ and $u_{k+1}$ be as those in Lemma
\ref{complexity-general-boundary} . Suppose that they satisfy the
decrease property
\begin{eqnarray*}
  &&\|u - u_{k+1} \|_{a,\Omega}^2 + \gamma_{\ast} osc^2_{k+1}(u_{k+1},\mathcal{T}_{k+1}) \nonumber\\
  & \leq&   \beta_{\ast}^2 \big(\|u-u_k\|_{a,\Omega}^2 + \gamma_{\ast} osc^2_k(u_k,\mathcal{T}_k)\big)
 \end{eqnarray*}
with constants $\gamma_{\ast}>0$  and $\beta_{\ast}\in
(0,\sqrt{\frac{1}{2}})$. Then the set
$\mathcal{R}:=R_{\mathcal{T}_k\rightarrow \mathcal{T}_{k+1}}$
satisfies the following inequality
\begin{eqnarray*}\label{optimal-marking-neq1}
\eta_k(u_k,\mathcal{R})\geq \hat{\theta}\eta_k(u_k,\mathcal{T}_k)
\end{eqnarray*}
with $\hat{\theta}^2 =
\frac{\tilde{C}_2(1-2\tilde{\beta_{\ast}}^2)}{\tilde{C}_0 (
\tilde{C}_1 + (1 +  2 C \tilde{C}_1) \tilde{\gamma_{\ast}})}$ and
$\tilde{C}_0 = \max(1, \frac{\tilde{C}_3}{\tilde{\gamma_{\ast}}})$,
where $\tilde{\beta_{\ast}}$ and $\tilde{\gamma_{\ast}}$ are defined
in (\ref{complexity-general-boundary-beta-gamma})  with $\delta_1$
being chosen such that $\tilde{\beta_{\ast}}^2\in (0,\frac{1}{2})$.
\end{corollary}
\begin{proof}
It is a direct consequence of combining $u_k=P_k w^k$ with Lemma
\ref{complexity-general-optimal-marking} and Lemma
\ref{complexity-general-boundary}.
\end{proof}

The key to relate the best mesh with AFEM triangulations is the fact
that procedure {\bf MARK} selects the marked set $\mathcal{M}_k$
with minimal cardinality.
\begin{lemma}
\textnormal{(Cardinality of $\mathcal{M}_k$).} Let $u\in
\mathcal{A}^s$, $\mathcal{T}_k$ be a conforming partition obtained
from $\mathcal{T}_0$, and $\theta$ satisfies $\theta\in (0,\frac{C_2
\gamma}{ C_3(C_1 + (1 + 2C C_1)\gamma )})$. Then the following
estimate is valid:
\begin{eqnarray}\label{DOF}
 \# \mathcal{M}_k \lc
 \left(\|u-u_k\|_{a,\Omega}^2 + \gamma osc^2_k(u_k, \mathcal{T}_k)\right)^{-1/2s} |u|_{s}^{1/s},
\end{eqnarray}
where the hidden constant  depends on the discrepancy between
$\theta$ and $\frac{C_2 \gamma}{C_3( C_1 + (1 + 2CC_1)\gamma) }$
with $C$ defined in Lemma \ref{complexity-general-optimal-marking}.
\end{lemma}
\begin{proof}
Let $\alpha, \alpha_1 \in (0,1)$ satisfy $\alpha_1\in (0,\alpha) $
and
 $$\theta < \frac{C_2 \gamma}{ C_3(C_1 + (1 + 2CC_1)\gamma)
 }(1-\alpha^2).$$
Choose $\delta_1\in (0,1)$ to satisfy
(\ref{error-reduction-neq-delta}) and
\begin{eqnarray}\label{thm-complexity-delta-cond-1}
(1+\delta_1)^2 \alpha_1^2 \leq \alpha^2,
\end{eqnarray}
which implies
\begin{eqnarray}\label{thm-complexity-delta-cond-3}
(1+\delta_1) \alpha_1^2 <1.
\end{eqnarray}

 Set
$$\varepsilon = \frac{1}{\sqrt{2}} \alpha_1 \big(\|u-u_k\|_{a,\Omega}^2 +
\gamma osc^2_k(u_k,
 \mathcal{T}_k)\big)^{1/2}$$ and let $\mathcal{T}_{\varepsilon}$ be a refinement of
 $\mathcal{T}_0$ with minimal degrees of freedom satisfying
 \begin{eqnarray}\label{complexity-optimal-neq0}
 \|u - u_{\varepsilon}\|_{a,\Omega}^2 + (\gamma + 1) osc^2_{\varepsilon}(u_{\varepsilon}, \mathcal{T}_{\varepsilon})  \leq
 \varepsilon^2.
 \end{eqnarray}
It follows from the definition of $\mathcal{A}^s$ that
  \begin{eqnarray*}\label{upper-bound-dof-neq1}
  ~~~~\#\mathcal{T}_{\varepsilon} - \# \mathcal{T}_0 \lc {\varepsilon}^{-1/s} |u|_{s}^{1/s}.
 \end{eqnarray*}
 Let $\mathcal{T}_{\ast}=\mathcal{T}_{\varepsilon}\oplus \mathcal{T}_k$ be the smallest  common refinement of
 $\mathcal{T}_k$ and $\mathcal{T}_{\varepsilon}$.
Note that $w^{\varepsilon}=K(\ell_{\varepsilon}
u_{\varepsilon}-Vu_{\varepsilon})$ satisfies
\begin{eqnarray*}\label{complexity-boundary-problem-2}
 L w^{\varepsilon} = \ell_{\varepsilon}
u_{\varepsilon}-Vu_{\varepsilon},
\end{eqnarray*}
we get from the definition of oscillation and Young's inequality
that
\begin{eqnarray*}
\widetilde{osc}^2_{\ast}(P_{\ast}w^{\varepsilon},\tau)
  &\leq&
  2\widetilde{osc}^2_{\ast}(P_{\varepsilon}w^{\varepsilon},\tau)+2C_0^2osc_{\ast}^2(\mathbf{A},\tau)\|P_{\varepsilon}w^{\varepsilon}-
  P_{\ast}w^{\varepsilon}\|^2_{1,\omega_{\tau}} ~~\forall \tau\in\mathcal
  {T}_{\ast},
\end{eqnarray*}
which together with the monotonicity property
$osc_{\ast}(\mathbf{A},\mathcal{T}_{\ast})\leq
osc_0(\mathbf{A},\mathcal{T}_0)$ yields
\begin{eqnarray*}
  \widetilde{osc}^2_{\ast}(P_{\ast}w^{\varepsilon},\mathcal{T}_{\ast})
  &\leq&
  2\widetilde{osc}^2_{\ast}(P_{\varepsilon}w^{\varepsilon},\mathcal{T}_{\ast})+2C\|P_{\varepsilon}w^{\varepsilon}-
  P_{\ast}w^{\varepsilon}\|^2_{a,\Omega},
\end{eqnarray*}
where $C=\Lambda_1 osc^2_0(\mathbf{A},\mathcal{T}_0)$. Due to the
orthogonality
\begin{eqnarray*}\label{ortho-relation}
\|w^{\varepsilon}-P_{\ast}w^{\varepsilon}\|^2_{a, \Omega} =
\|w^{\varepsilon} - P_{\varepsilon}w^{\varepsilon}\|^2_{a, \Omega}
  - \|P_{\ast}w^{\varepsilon} -
  P_{\varepsilon}w^{\varepsilon}\|_{a, \Omega}^2,
\end{eqnarray*}
we arrive at
\begin{eqnarray*}
  &&\| w^{\varepsilon} - P_{\ast}
  w^{\varepsilon}\|_{a,\Omega}^2 + \frac{1}{2 C}\widetilde{osc}^2_{\ast}(P_{\ast}w^{\varepsilon},\mathcal{T}_{\ast})\nonumber\\
  &\leq & \|w^{\varepsilon} - P_{\varepsilon}
  w^{\varepsilon}\|_{a,\Omega}^2+ \frac{1}{C}
  osc^2_{\varepsilon}(P_{\varepsilon}w^{\varepsilon},\mathcal{T}_{\varepsilon}).
 \end{eqnarray*}
Since (\ref{gamma-boundary}) implies $\tilde{\gamma} \leq \frac{1}{2
C}$,  we obtain that
\begin{eqnarray*}
  &&\| w^{\varepsilon} - P_{\ast}
  w^{\varepsilon}\|_{a,\Omega}^2 + \tilde{\gamma}\widetilde{osc}^2_{\ast}(P_{\ast}w^{\varepsilon},\mathcal{T}_{\ast})\nonumber\\
  &\leq &  \|w^{\varepsilon} - P_{\varepsilon} w^{\varepsilon}\|_{a,\Omega}^2
   + \frac{1}{ C}
   osc^2_{\varepsilon}(P_{\varepsilon}w^{\varepsilon},\mathcal{T}_{\varepsilon})\nonumber\\
  &\leq & \|w^{\varepsilon} - P_{\varepsilon}
  w^{\varepsilon}\|_{a,\Omega}^2+ (\tilde{\gamma} + \sigma)
  osc^2_{\varepsilon}(P_{\varepsilon}w^{\varepsilon},\mathcal{T}_{\varepsilon})
 \end{eqnarray*}
 with $\sigma = \frac{1}{C} - \tilde{\gamma} \in (0, 1) $.
Applying the similar argument in the proof of Theorem
\ref{error-reduction} when (\ref{lemma-bound-general-conc-3}) is
replaced by (\ref{lemma-bound-general-conc-2}), we then get
 \begin{eqnarray}\label{complexity-optimal-neq2}
 &&\|u - u_{\ast}\|_{a,\Omega}^2 + \gamma osc^2_{\ast}(u_{\ast},
 \mathcal{T}_{\ast})\nonumber\\
 & \leq& \alpha_0^2\left(
 \|u - u_{\varepsilon}\|_{a,\Omega}^2 + (\gamma + \sigma)osc^2_{\varepsilon}
 (P_{\varepsilon}w^{\varepsilon},\mathcal{T}_{\varepsilon})\right) \nonumber\\
&\leq &  \alpha_0^2\left(\|u - u_{\varepsilon}\|_{a,\Omega}^2
     + (\gamma + 1)
osc^2_{\varepsilon}
 (P_{\varepsilon}w^{\varepsilon},\mathcal{T}_{\varepsilon})\right),
 \end{eqnarray}
 where
 \begin{eqnarray*}
\alpha_0^2 := \frac{ (1+ \delta_1)  +  C_4
 \tilde{\kappa}(h_0) }{1 - C_4 \delta_1^{-1}
 \tilde{\kappa}^2(h_0)}
\end{eqnarray*}
and $C_4$ is the constant appearing in the proof of Theorem
\ref{error-reduction}. Thus, by (\ref{complexity-optimal-neq0}) and
(\ref{complexity-optimal-neq2}), it follows
\begin{eqnarray*} \|u - u_{\ast}\|_{a,\Omega}^2 + \gamma osc^2_{\ast}(u_{\ast},
 \mathcal{T}_{\ast}) \leq \check{\alpha}^2 \big(\|u-u_k\|^2_{a,\Omega} + \gamma osc^2_k(u_k,
 \mathcal{T}_k)\big)
\end{eqnarray*}
with  $\check{\alpha} = \frac{1}{\sqrt{2}} \alpha_0 \alpha_1$.  In
view of (\ref{thm-complexity-delta-cond-3}), we have
$\check{\alpha}^2\in (0,\frac{1}{2})$ when $h_0\ll 1$. Let $\mathcal
{R}:=R_{\mathcal{T}_k\rightarrow \mathcal{T}_{\ast}}$,  by Corollary
\ref{optimal-marking}, we have that
$\mathcal{T}_{\ast}$ satisfies 
\begin{eqnarray*}
\eta_k(u_k, \mathcal {R}) \geq \check{\theta}
\eta_k(u_k,\mathcal{T}_k),
\end{eqnarray*}
where $\check{\theta}^2 =
\frac{\tilde{C}_2(1-2\hat{\alpha}^2)}{\tilde{C}_0 ( \tilde{C}_1 + (1
+ 2 C \tilde{C}_1)\hat{\gamma})}, \quad \hat{\gamma}=
\frac{\gamma}{1 - C_5 \delta_1^{-1} \tilde{\kappa}^2(h_0)}$,
$\tilde{C}_0 = \max(1, \frac{\tilde{C}_3}{\hat{\gamma}})$, and
\begin{eqnarray*}
 \hat{\alpha}^2= \frac{ (1+ \delta_1)\check{\alpha}^2  +  C_5
 \tilde{\kappa}(h_0) }{1 - C_5 \delta_1^{-1}
 \tilde{\kappa}^2(h_0)}.
 \end{eqnarray*}
It follows from the
 definition of $\gamma$ (see (\ref{gamma})) and $\tilde{\gamma}$ (see (\ref{gamma-boundary}))
  that $\hat{\gamma}<1$ and hence ${\tilde C}_0 =
\frac{\tilde{C}_3}{\hat{\gamma}}.$
 Since $h_0 \ll 1$,
  we obtain that  $\hat{\gamma}>\gamma$ and $\hat{\alpha}\in
  (0,\frac{1}{\sqrt{2}}\alpha)$ from (\ref{thm-complexity-delta-cond-1}).
It is easy to see from (\ref{coef-eigen-bound}) and
$\hat{\gamma}>\gamma$ that
\begin{eqnarray*}
 &&\check{\theta}^2=\frac{\tilde{C}_2(1-2\hat{\alpha}^2)}{\frac{\tilde{C}_3}{\hat{\gamma}}
(\tilde{C}_1 + (1 + 2 C\tilde{C}_1)\hat{\gamma} )}
 \geq \frac{\tilde{C}_2}{\tilde{C}_3(
\frac{\tilde{C}_1}{\hat{\gamma}} + 1 + 2 C\tilde{C}_1
)}(1-\alpha^2)\nonumber\\
&=&
\frac{\frac{C_2}{(1-\tilde{C}\tilde{\kappa}(h_0))^2}}{\frac{C_3}{(1-\tilde{C}\tilde{\kappa}(h_0))^2}
(\frac{C_1}{\hat{\gamma}((1+\tilde{C}\tilde{\kappa}(h_0))^2)}+1
+2C\frac{C_1}{(1+\tilde{C}\tilde{\kappa}(h_0))^2})}(1-\alpha^2)
\nonumber\\&\geq& \frac{C_2 }{C_3 ( \frac{C_1}{\gamma} + (1 + 2 C
C_1) )}(1-\alpha^2) = \frac{C_2 \gamma }{C_3 ( C_1 + (1 + 2 C
C_1)\gamma )}(1-\alpha^2)
> \theta
\end{eqnarray*}
when $h_0 \ll 1$. Thus
\begin{eqnarray*}
 \#\mathcal{M}_k &\leq&
 \#\mathcal{R}
  \leq \#\mathcal{T}_{\ast} - \#\mathcal{T}_k
\leq \#\mathcal{T}_{\varepsilon}- \#\mathcal{T}_0 \nonumber\\
 &\leq&(\frac{1}{\sqrt{2}}\alpha_1)^{-1/s}
  \left(\|u - u_k\|_{a,\Omega}^2 + \gamma osc^2_k(u_k,
 \mathcal{T}_k)\right) ^{-1/2s} |u|_{s}^{1/s},
\end{eqnarray*}
which is the desired estimate (\ref{DOF}) with an explicit
dependence on the discrepancy between $\theta$ and $\frac{C_2
\gamma}{C_3( C_1 + (1 + 2CC_1)\gamma )}$ via $\alpha_1$. This
completes the proof.
\end{proof}

As a consequence, we obtain the optimal complexity as follows.
\begin{theorem}\label{thm-optimal-complexity}
Let $u \in  \mathcal{A}^s$ and $\{u_k\}_{k\in  \mathbb{N}_0}$ be a
sequence of finite element
 solutions corresponding to a sequence of nested finite element spaces $\{{S^k_0(\Omega)}\}_{k\in  \mathbb{N}_0}$
 produced by  {\bf Algorithm $D$}.
 Then
 \begin{eqnarray*}
  \|u-u_k\|_{a,\Omega}^2 + \gamma osc^2_k(u_k, \mathcal{T}_k) \lc
  (\#\mathcal{T}_k
  -\#\mathcal{T}_0)^{-2s}|u|_s^2,
 \end{eqnarray*}
 where the hidden constant depends on the exact solution u and
 the discrepancy between
 $\theta$ and $\frac{C_2 \gamma}{C_3( C_1 + (1 + 2 C C_1)\gamma)
}$.
\end{theorem}
\begin{proof}
It follows from (\ref{complexity}) and (\ref{DOF}) that
\begin{eqnarray*}
& &\#\mathcal{T}_k - \#\mathcal{T}_0 \lc \sum_{j=0}^{k-1}
\#\mathcal{M}_j \nonumber\\
 &\lc& \sum_{j=0}^{k-1}\left(\|u - u_j\|_{a,\Omega}^2 + \gamma osc^2_j(u_j,
\mathcal{T}_j)\right)^{-1/2s}|u|_{s}^{1/s}.
\end{eqnarray*}
Note that (\ref{lower-bound}) implies
\begin{eqnarray*}
\|u - u_j\|^2_{a, \Omega} + \gamma \eta^2_j(u_j, \mathcal{T}_j) \leq
\check{C} \big( \|u - u_j\|^2_{a, \Omega} + \gamma osc^2_j(u_j,
\mathcal{T}_j)\big),
\end{eqnarray*}
where $ \check{C} = \max(1 + \frac{\gamma}{C_2}, \frac{C_3}{C_2}).$
It then turns out
\begin{eqnarray*}
\#\mathcal{T}_k - \#\mathcal{T}_0 &\lc& \sum_{j=0}^{k-1}
 \left(\|u - u_j\|_{a,\Omega}^2 + \gamma \eta^2_j(u_j,
\mathcal{T}_j)\right)^{-1/2s}|u|_{s}^{1/s}.
\end{eqnarray*}
Due to  (\ref{error-reduction-neq1}), we obtain  for $0\leq j < k$
that
\begin{eqnarray*}
 \|u-u_k\|_{a,\Omega}^2 + \gamma \eta^2_k(u_k, \mathcal{T}_k)\leq
   \xi^{2(k-j)} \left( \|u-u_j\|_{a,\Omega}^2 + \gamma \eta^2_j(u_j,
   \mathcal{T}_j)\right).
\end{eqnarray*}
Consequently,
 \begin{eqnarray*}
 \#\mathcal{T}_k-\#\mathcal{T}_0
 &\lc& |u|_{s}^{1/s} \left(\|u-u_k\|_{a,\Omega}^2 + \gamma \eta^2_k(u_k, \mathcal{T}_k)\right)^{-1/2s}
 \sum_{j=0}^{k-1}\xi^{\frac{k-j}{s}} \nonumber\\
 &\lc& |u|_{s}^{1/s} \left(\|u-u_k\|_{a,\Omega}^2 + \gamma \eta^2_k(u_k, \mathcal{T}_k)\right)^{-1/2s},
\end{eqnarray*}
 the last inequality holds because of the fact $\xi<1$.

 Since $osc_k(u_k, \mathcal{T}_k) \leq  \eta_k(u_k,\mathcal{T}_k)$,
 we
 arrive at
\begin{eqnarray*}
 \#\mathcal{T}_k-\#\mathcal{T}_0
 \lc \left(\|u-u_k\|_{a,\Omega}^2 +
 \gamma osc_k^2(u_k, \mathcal{T}_k)\right)^{-1/2s}|u|_{s}^{1/s}.
\end{eqnarray*}
 This completes the proof.
\end{proof}

\section{Applications}
In this section, we provide three typical examples to show that our
general theory is quite useful.

\subsection{A nonsymmetric problem}\label{nonsymmetric}
The first example is a nonsymmetric elliptic partial differential
equation of second order. We consider the following problem: Find $u
\in H^1_0(\Omega)$ such that
\begin{eqnarray}\label{application2}
\left\{\begin{array}{rl}
 -\nabla \cdot (\mathbf{A}\nabla u)+ {\bf b} \cdot \nabla u + cu&=f \,\,\, \mbox{in} \quad \Omega,\\
u &= 0\,\,\,\mbox{on}~~\partial\Omega,
\end{array}\right.
\end{eqnarray}
where $\Omega \subset \mathbb{R}^d (d\ge 2)$ is a bounded ploytopic
domain, $\mathbf{A}:\Omega\rightarrow  \mathbb{R}^{d\times d}$ is
piecewise Lipschitz over initial triangulation $\mathcal {T}_{0}$,
for $x\in \Omega$  matrix $\mathbf{A(x)}$ is symmetric and positive
definite with smallest eigenvalue uniformly bounded away from 0,
${\bf b} \in [L^{\infty}(\Omega)]^d $ is divergence free , $c \in
L^{\infty}(\Omega),$ and $f \in L^2(\Omega)$ .


A finite element discretization of (\ref{application2}) reads: Find
$u_h \in S^h_0(\Omega)$ satisfying
\begin{eqnarray}\label{variation-2}
(\mathbf{A}\nabla u_h,\nabla v)+({\bf b}\cdot \nabla
u_h,v)+(cu_h,v)= (f,v)~~~~~\forall v \in S^h_0(\Omega).
\end{eqnarray}
It is seen that (\ref{variation-2}) is a special case of
(\ref{Gdis-fem}), in which $Vu:={\bf b} \cdot \nabla u + cu$ and
$\ell u=\ell_h u_h=f$.  Consequently, $\kappa_1(h)=0$,
$w^h=K(f-Vu_h)$ and
\begin{eqnarray*}
u-w^h = KV(u_h-u) = KV(I-P_h)(u_h-u).
\end{eqnarray*}

 Obviously, $V:H^1_0(\Omega)\rightarrow L^2(\Omega)$ is a
linear bounded operator and $KV$ is a compact operator over
$H^1_0(\Omega)$. We have the conclusion of Theorem
\ref{thm-general-boundary}.

In this application, the element residual and jump residual become
\begin{eqnarray*}
  \mathcal{R}_{\tau}(u_h) &:=& f -{\bf b} \cdot \nabla u_h - cu_h+\nabla \cdot (\mathbf{A}\nabla u_h)~~~~ \mbox{in}~ \tau\in
  \mathcal{T}_h,\\
  J_e(u_h) &:=&[[\mathbf{A}\nabla u_h]]_e \cdot \nu_e  ~~~~~~~~~~~~~~~~~~~ \mbox{on}~ e\in
  \mathcal{E}_h
\end{eqnarray*}
while the corresponding error estimator $\eta_h(u_h, \mathcal{T}_h)$
and the oscillation $osc_h(u_h,\mathcal{T}_h)$ are defined by
(\ref{Gerror-estimator}) and (\ref{Goscilliation}), respectively.
Thus Theorem \ref{error-reduction} and Theorem
\ref{thm-optimal-complexity} ensure the convergence and optimal
complexity of AFEM for nonsymmetric problem (\ref{application2}).

\subsection{A nonlinear problem}\label{nonlinear}
In this subsection, we derive the convergence and optimal complexity
of AFEM for a nonlinear problem from our general theory.

Consider the following nonlinear problem: Find $u \in H^1_0(\Omega)$
such that
 \begin{eqnarray}\label{application3}
\left\{\begin{array}{rl}
 \mathcal {L}u:=-\Delta u+f(x ,u)&=0 \,\,\, \mbox{in} \quad \Omega,\\
u &= 0\,\,\,\mbox{on}~~\partial\Omega,
\end{array}\right.
\end{eqnarray}
where $f(x,y)$ is a smooth function on
$\mathbb{R}^3\times\mathbb{R}^1$.

 For convenience, we shall drop the dependence of variable $x$ in
$f(x,u)$ in the following exposition. We  assume that $u\in
H^1_0(\Omega)\cap H^{1+s}(\Omega)$ for some $s\in (0,1]$.
 For any $w\in
H^1_0(\Omega)\cap H^{1+s}(\Omega)$, the linearized operator
$\mathcal {L}'_w$ at $w$ (namely, the Fr{\' e}chet derivative of
$\mathcal {L}$ at $w$) is then given by
\begin{eqnarray*} \mathcal {L}'_w=-\Delta + f'(w).
\end{eqnarray*}
We assume that $\mathcal {L}'_w:H^1_0(\Omega)\rightarrow
H^{-1}(\Omega)$ is an isomorphism. As a result, $u\in
H^1_0(\Omega)\cap H^{1+s}(\Omega)$ must be an isolated solution of
(\ref{application3}). The associated finite element scheme for
(\ref{application3}) reads: Find $u_h\in S^h_0(\Omega)$ satisfying
\begin{eqnarray}\label{variation-3}
(\nabla u_h,\nabla v)+(f(u_h),v)=0~~~~\forall v\in S^h_0(\Omega).
\end{eqnarray}
 Let $a(\cdot ,\cdot )=(\nabla \cdot,\nabla \cdot)$, $K=(-\Delta)^{-1}: L^2(\Omega)\rightarrow H^1_0(\Omega)$,
 $V=0$ and $\ell_h w = -f(w)$ for  any $w\in S^h_0(\Omega) $, then (\ref{variation-3}) becomes
 (\ref{Gdis-fem}).

As usual, to analyze the finite element approximation of nonlinear
problem (\ref{variation-3}), we require  mesh $\mathcal {T}_h$ to
satisfy that there exists $\varsigma\ge 1$ such that (c.f.
\cite{xu-zhou-01})
\begin{eqnarray*}
h^{\varsigma}\lc h(x) ~~ x\in\Omega,
\end{eqnarray*}
where $h(x)$ is the diameter $h_{\tau}$ of the element $\tau$
containing $x$. We consider the case of that $S^h_0(\Omega)$ is the
conforming piecewise linear finite element space associated with
$\mathcal {T}_h$. We assume that $\varsigma<2s$. Thus we can choose
$p\in (3, 6\varsigma/(3\varsigma-2s)]$ and obtain from Theorem 3.1
and Theorem 3.2 of \cite{xu-zhou-01} that
\begin{lemma}\label{fem-nonlinear-lemma}
If $h\ll 1$, then
\begin{eqnarray*}
\|u-u_h\|_{1,\Omega}+h^s\|u_h\|_{0,\infty,\Omega}\lc h^s
\end{eqnarray*}
and
\begin{eqnarray*}
\|u-u_h\|_{0,\Omega}\lc r(h)\|u-u_h\|_{1,\Omega},
\end{eqnarray*}
where  $r(h)\rightarrow 0$ as $h\rightarrow 0$.
\end{lemma}

Now we shall show that Theorem 3.1 is applicable for
(\ref{application3}). Since $K$ is monotone and $f(x,y)$ is smooth,
we have from Lemma \ref{fem-nonlinear-lemma} that
\begin{eqnarray*}
& &\|K(f(u)-f(u_h))\|_{a,\Omega} \lc \|K(u-u_h)\|_{a,\Omega}
\nonumber\\&\lc & \|u-u_h\|_{0,\Omega}  \lc
r(h)\|u-u_h\|_{a,\Omega}.
\end{eqnarray*}
Therefore we have (\ref{general-boundary-neq}) when we choose
$\kappa_1(h)=r(h)$  and $\kappa_2(h)=0$.

In this application,  the element residual and jump residual become:
\begin{eqnarray*}
  \mathcal{R}_{\tau}(u_h) &:=& -f(u_h)+\Delta u_h  \qquad \mbox{in}~ \tau\in
  \mathcal{T}_h,\\
  J_e(u_h) &:=&  -\nabla u_h^{+}\cdot \nu^{+} -  \nabla u_h^{-}\cdot
  \nu^{-} := [[\nabla u_h]]_e \cdot \nu_e  ~~~~ \mbox{on}~ e\in
  \mathcal{E}_h
\end{eqnarray*}
and the corresponding error estimator $\eta_h(u_h, \mathcal{T}_h)$
and the oscillation $osc_h(u_h,\mathcal{T}_h)$ are defined by
(\ref{Gerror-estimator}) and (\ref{Goscilliation}), respectively.
Then Theorem \ref{error-reduction} and Theorem
\ref{thm-optimal-complexity} ensure the convergence and  optimal
complexity of AFEM for  nonlinear problem (\ref{application3}).

\subsection{An unbounded coefficient problem}\label{nonsmooth} Finally,
we investigate a nonlinear eigenvalue problem, of which a
coefficient is unbounded. It is known that electronic structure
computations require solving the following Kohn-Sham equations
\cite{Beck-00,gong-shen-zhang-zhou-08,Kohn-Sham-65}
\begin{eqnarray}\label{KS}
\left(-\frac{1}{2}\Delta-\sum_{j=1}^{N_{atom}}\frac{Z_j}{|x-r_j|}+\int_{\mathbb{R}^3}\frac{\rho(y)}{|x-y|}d
y+V_{xc}(\rho)\right)u_i=\lambda_i u_i ~~in \quad \mathbb{R}^3,
\end{eqnarray}
where $N_{atom}$ is the total number of atoms in the system, $Z_j$
is the valance charge of this ion (nucleus plus core electrons),
$r_j$ is the position of the $j$-th atom $(j=1,\cdots ,N_{atom})$,
$$
\rho=\sum_{i=1}^{N_{occ}}c_i|u_i|^2
$$
with $u_i$ the $i$-th smallest eigenfunction, $c_i$ the number of
electrons on the i-th orbit, and $N_{occ}$ the total number of the
occupied orbits.
The central computation in solving the Kohn-Sham
equation is the repeated solution of the following eigenvalue
problem: Find $(\lambda,u)\in \mathbb{R}\times H^1_0(\Omega)$ such
that
\begin{eqnarray}\label{application1}
\left\{\begin{array}{rl}
 -\frac{1}{2}\Delta u+ Vu &= \lambda u \quad \mbox{in}~~ \Omega,\\
\|u\|_{0,\Omega} &= 1,
\end{array}\right.
\end{eqnarray}
where $\Omega$ is a bounded domain in $\mathbb{R}^3$, $V=V_{ne}+V_0$
is the so-called effective potential. Here, $V_0 \in
L^{\infty}(\Omega)$ and
\begin{eqnarray*}
V_{ne}(x)=-\sum_{j=1}^{N_{atom}}\frac{Z_j}{|x-r_j|}.
\end{eqnarray*}


 A finite element discretization of
(\ref{application1}) reads: Find $(\lambda_h,u_h)\in
\mathbb{R}\times S^h_0(\Omega)$ such that
\begin{eqnarray}\label{variation-1}
\frac{1}{2}(\nabla u_h,\nabla
v)+(Vu_h,v)=\lambda_h(u_h,v)~~~~~\forall v \in S^h_0(\Omega).
\end{eqnarray}
Let $\ell_h : S^h_0(\Omega)\rightarrow L^2(\Omega)$ be defined by
\begin{eqnarray*}
\ell_h v=\lambda_h v ~~~~\forall v \in S^h_0(\Omega),
\end{eqnarray*}
then (\ref{variation-1}) is a special case of (\ref{Gdis-fem})
 when $a(\cdot,\cdot)=\frac{1}{2}(\nabla \cdot, \nabla \cdot)$ and $K=\frac{1}{2}(-\Delta)^{-1}:
L^2(\Omega)\rightarrow H^1_0(\Omega)$.

Using the uncertainty principle lemma (see, e.g.,
\cite{Reed-Simon-75})
\begin{eqnarray*}
\int_{\mathbb{R}^3} \frac{w^2(x)}{|x|^2}\leq
4\int_{\mathbb{R}^3}|\nabla w|^2~~~~\forall w \in
C^{\infty}_{0}(\mathbb{R}^3)
\end{eqnarray*}
and the fact that $C^{\infty}_0(\Omega)$ is dense in
$H^1_0(\Omega)$, we obtain
\begin{eqnarray*}
\int_{\Omega} \frac{w^2(x)}{|x|^2}\leq 4\int_{\Omega}|\nabla
w|^2~~~~\forall w \in H^1_0(\Omega).
\end{eqnarray*}
Then for any $w \in H^1_0(\Omega)$, we have
\begin{eqnarray*}
\|V_{ne} w+V_0w\|_{0,\Omega}&\leq& C\|w\|_{1,\Omega},
\end{eqnarray*}
namely, $V$ is a bounded operator over $H^1_0(\Omega)$. Thus $KV$ is
a compact operator over $H^1_0(\Omega)$.

We consider the case of that $(\lambda,u)\in \mathbb{R}\times
H^1_0(\Omega)$ is some simple eigenpair of (\ref{application1}) with
$\|u\|_{0,\Omega}=1$. Note that for $\ell v:=\lambda v ~~\forall
v\in H^1_0(\Omega)$, there holds
\begin{eqnarray*}
K(\ell u-\ell_h u_h) = \lambda K(u-u_h)+(\lambda-\lambda_h)Ku_h.
\end{eqnarray*}
So if $(\lambda_h,u_h)\in \mathbb{R}\times S^h_0(\Omega)$ is the
associated finite element  eigenpair of (\ref{variation-1}) with
$\|u_h\|_{0,\Omega}=1$ that satisfy $$
\|u-u_h\|_{0,\Omega}+|\lambda-\lambda_h|\lc
\kappa_1(h)\|u-u_h\|_{a,\Omega}, $$ we then have (c.f.
\cite{Dai-Xu-Zhou-08})
\begin{eqnarray*}
 \| K(\ell u-\ell_h u_h)\|_{a,\Omega}
 = O(\kappa_1(h))\|u - u_h\|_{a,\Omega},
 \end{eqnarray*}
where $\kappa_1(h):= \rho_{_{\Omega}}(h)+\|u-u_h\|_{a,\Omega}$
satisfying $\kappa_1(h)\rightarrow 0$ as $h\rightarrow 0$.

In this application,  the element residual and jump residual become:
\begin{eqnarray*}
  \mathcal{R}_{\tau}(u_h) &:=&\lambda_hu_h-Vu_h +\frac{1}{2}\Delta u_h~~~~ \mbox{in}~ \tau\in
  \mathcal{T}_h,\\
  J_e(u_h) &:=& [[\frac{1}{2}\nabla u_h]]_e \cdot \nu_e  ~~~~~~~~~~~~~~ \mbox{on}~ e\in
  \mathcal{E}_h
\end{eqnarray*}
and the corresponding error estimator $\eta_h(u_h, \mathcal{T}_h)$
and the oscillation $osc_h(u_h,\mathcal{T}_h)$ are defined by
(\ref{Gerror-estimator}) and (\ref{Goscilliation}), respectively.
Then Theorem \ref{error-reduction} and Theorem
\ref{thm-optimal-complexity}  ensure the convergence and optimal
complexity of AFEM for  unbounded coefficient problem
(\ref{application1}) (c.f. \cite{Dai-Xu-Zhou-08}).

\section{Numerical examples}
In this section we will report some numerical results to illustrate
our theory. Our numerical results were carried out on LSSC-II in the
State Key Laboratory of Scientific and Engineering Computing,
Chinese Academy of Sciences, and our codes were based on the toolbox
PHG of the State Key Laboratory of Scientific and Engineering
Computing, Chinese Academy of Sciences.


{\bf Example 1}. We consider (\ref{application2}) when the
homogenous Dirichlet boundary condition is replaced by $u = g$ on
$\partial \Omega$ and  $\Omega = (0,1)^3$  with the isotropic
diffusion coefficient $\mathbf{A} = \epsilon I$, $\epsilon =
10^{-2}$, convection velocity $\mathbf{b} = (2,3,4)$, and c = 0
(c.f. \cite{Knobloch-Tobiska-03} for a 2D case and Remark
\ref{remark-boundary}). The exact solution is given by
\begin{eqnarray*}
u=\left(x^3-\exp\big(\frac{2(x-1)}{\epsilon}\big)\right)
\left(y^2-\exp\big(\frac{3(y-1)}{\epsilon}\big)\right)
\left(z-\exp\big(\frac{4(z-1)}{\epsilon}\big)\right).
\end{eqnarray*}
For small $\epsilon >0$ the solution has the typical layer behavior
in the neighbourhood of $x=1$, $y=1$, $z=1$, respectively.
 The Dirichlet boundary condition $g(x,y,z)$ on $\partial
\Omega$ is given by
\begin{equation*}\label{boundary}
g(x,y,z)=
\left\{\begin{array}{rcl} \displaystyle 0  ~~~~~~~~ x=1~~ or~~ y=1 ~~or~~ z=1, \\
u(x,y,z) ~~~~ x=0~~ or~~ y=0~~ or~~ z=0.
\end{array}\right.
\end{equation*}

\begin{figure}[htbp]
\begin{center}
\setlength{\unitlength}{1cm}
\begin{minipage}[t]{5.0cm}
\begin{picture}(5.0,5.0)\resizebox*{5cm}{5cm}
{\includegraphics{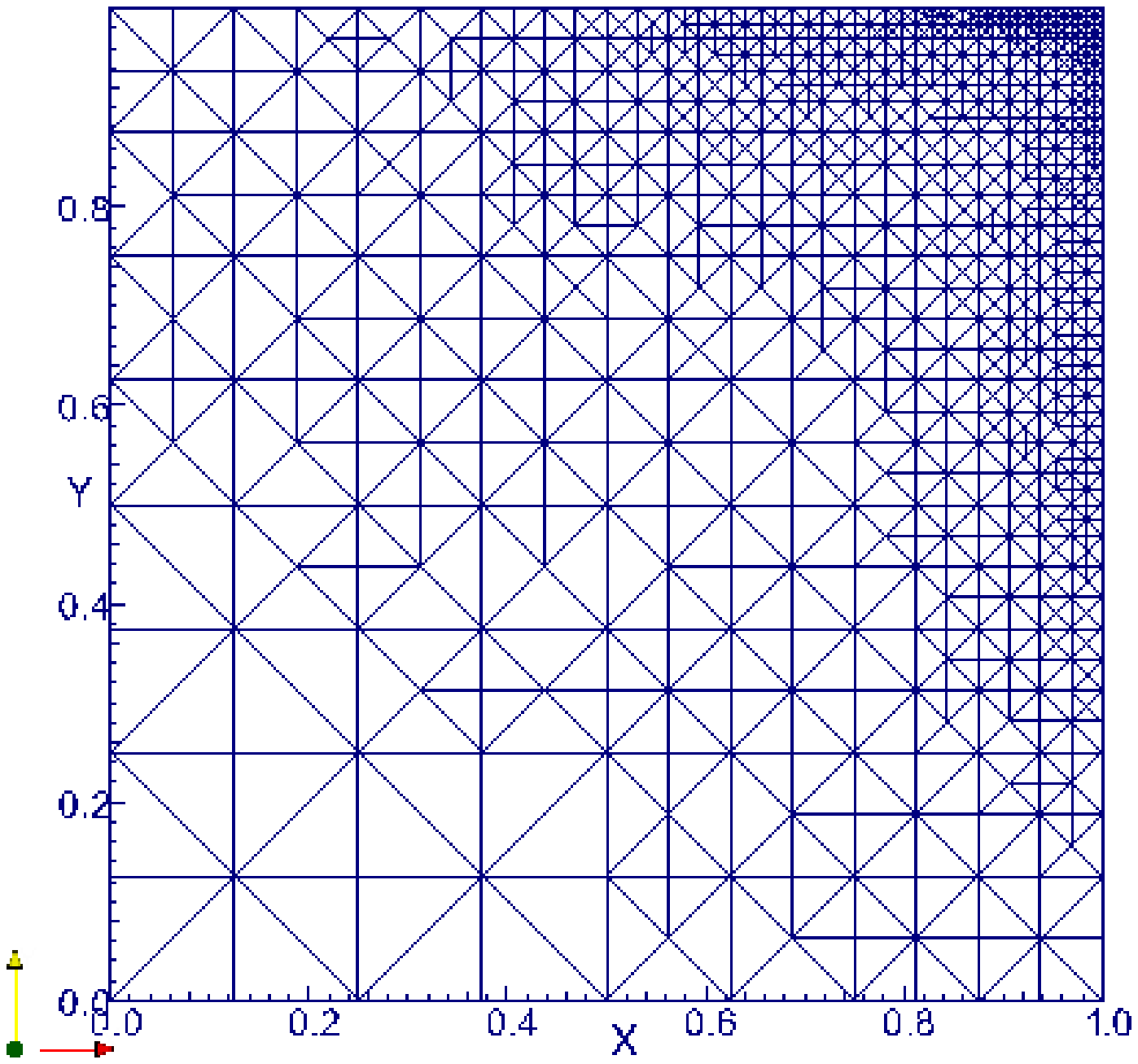}}\end{picture}
\begin{center} Z=0.0 \end{center}\caption{The cross-section of an
adaptive mesh of {\bf Example 1} using linear finite
elements}\label{general_P1_mesh}
\end{minipage}
\hfill
\begin{minipage}[t]{5.0cm}
\begin{picture}(5.0,5.0)\resizebox*{5cm}{5cm}
{\includegraphics{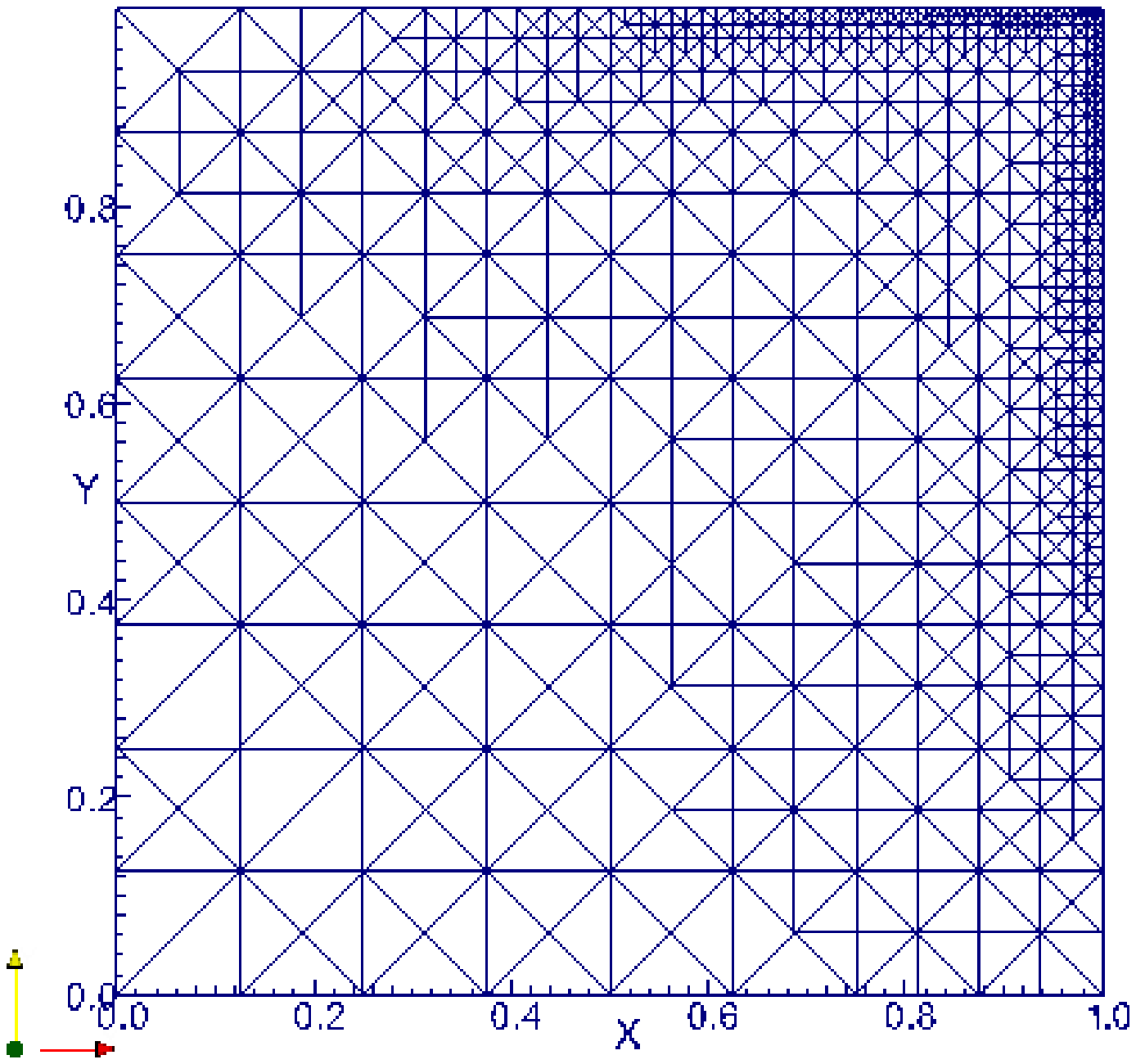}}\end{picture}
\begin{center} Z=0.0 \end{center}\caption{The  cross-section of an
adaptive mesh of {\bf Example 1} using quadratic finite
elements}\label{general_P2_mesh}
\end{minipage}
\end{center}
\end{figure}

\begin{figure}[htbp]
\begin{center}
\setlength{\unitlength}{1cm}
\begin{minipage}[t]{5.0cm}
\begin{picture}(5.0,5.0) \resizebox*{5cm}{5cm}{\includegraphics{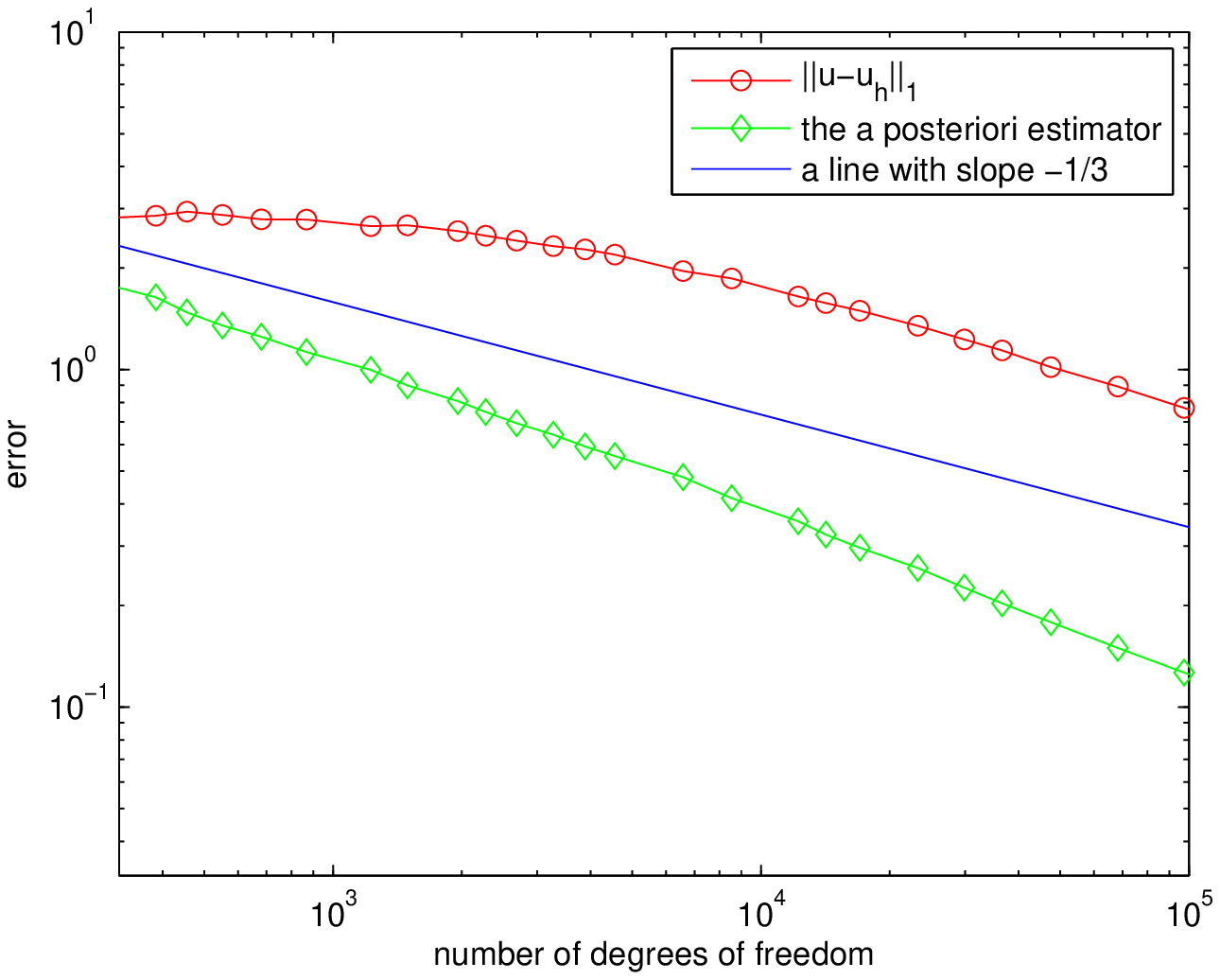}}\end{picture}
 \caption{The convergence curves
  of {\bf Example 1}
 using linear finite elements}\label{general_P1}
\end{minipage}
\hfill
\begin{minipage}[t]{5.0cm}
\begin{picture}(5.0,5.0)\resizebox*{5cm}{5cm} {\includegraphics{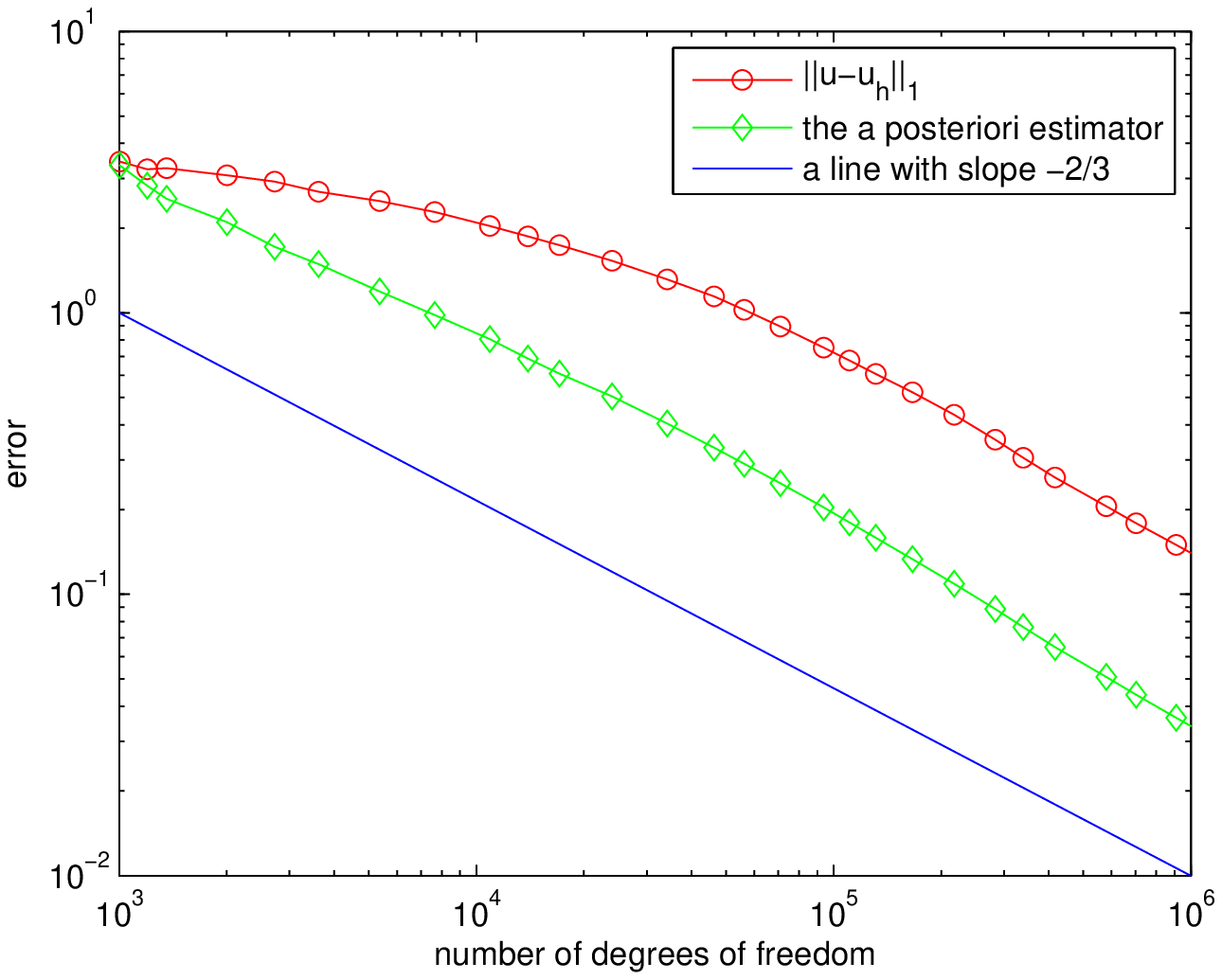}}\end{picture}
 \caption{The convergence curves
  of {\bf Example 1}
 using quadratic finite elements }\label{general_P2}
\end{minipage}
\end{center}
\end{figure}

Some adaptively refined meshes
are displayed in Fig. \ref{general_P1_mesh} and Fig.
\ref{general_P2_mesh}. Our numerical results are presented in Fig.
\ref{general_P1} and Fig. \ref{general_P2}. It is shown from Fig.
\ref{general_P2} that  $\|u - u_h\|_1$ is proportional to the a
posteriori error estimators, which indicates the efficiency of the a
posteriori error estimators given in section \ref{nonsymmetric}.
Besides, it is also seen from Fig. \ref{general_P1} and Fig.
\ref{general_P2}  that, by using linear finite elements and
quadratic finite elements, the convergence curves of  errors are
approximately parallel to the line with slope $-1/3$ and the line
with slope $-2/3$, respectively. These mean that the approximation
error of the exact solution has optimal convergence rate, which
coincides with our theory in section \ref{section-convergence}.

{\bf Example 2}. Consider the following nonlinear problem:
\begin{eqnarray*}\label{example3}
\left\{\begin{array}{rl}
 -\Delta u+ u^3&=f \,\,\, \mbox{in} \quad \Omega,\\
u &= 0\,\,\,\mbox{on}~~\partial\Omega,
\end{array}\right.
\end{eqnarray*}
where $\Omega = (0,1)^3$. The exact solution is given by $u=\sin(\pi
x_1)\sin(\pi x_2)\sin(\pi x_3)/(x_1^2+x_2^2+x_3^2)^{1/2}$.

\begin{figure}[htbp]
\begin{center}
\setlength{\unitlength}{1cm}
\begin{minipage}[t]{5.0cm}
\begin{picture}(5.0,5.0) \resizebox*{5.0cm}{5.0cm}{\includegraphics{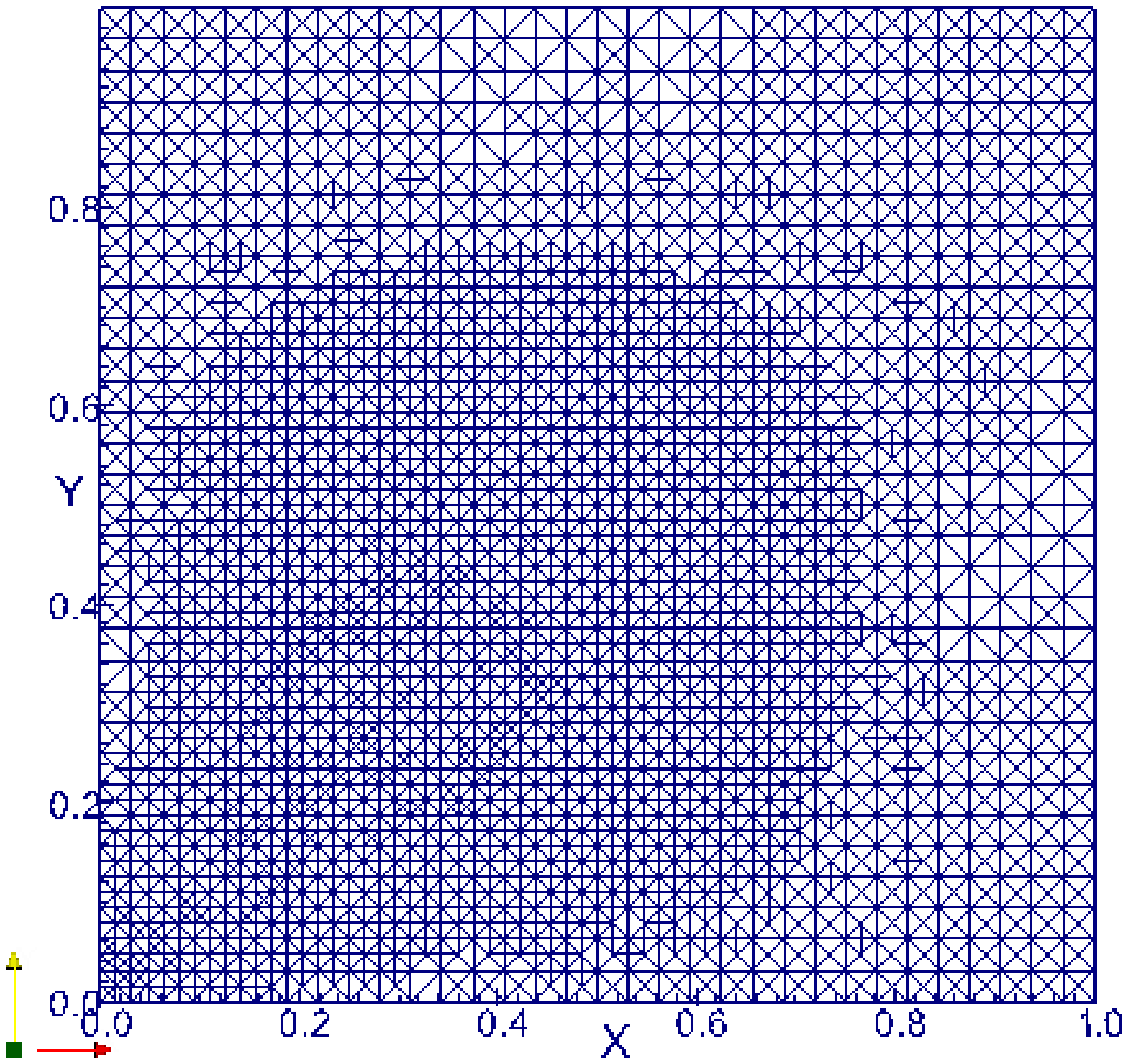}}\end{picture}
 \caption{The  cross-section of an adaptive mesh of {\bf Example 2} using linear finite
elements}\label{nonlinear_P1_mesh}
\end{minipage}
\hfill
\begin{minipage}[t]{5.0cm}
\begin{picture}(5.0,5.0)\resizebox*{5.0cm}{5.0cm} {\includegraphics{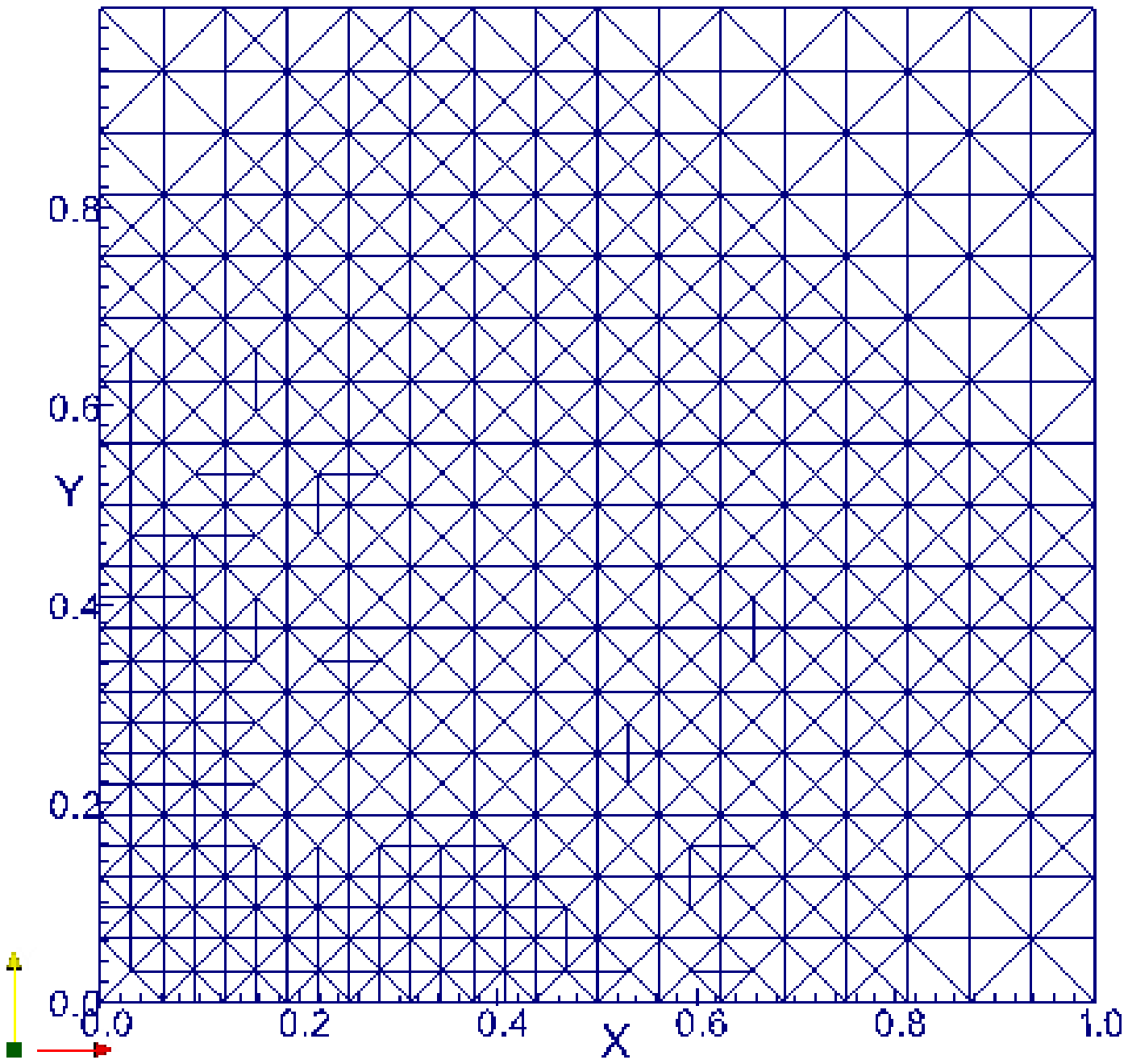}}\end{picture}
\caption{The  cross-section of an adaptive mesh of {\bf Example 2}
using quadratic finite elements}\label{nonlinear_P2_mesh}
\end{minipage}
\end{center}
\end{figure}

\begin{figure}[htbp]
\begin{center}
\setlength{\unitlength}{1cm}
\begin{minipage}[t]{5.0cm}
\begin{picture}(5.0,5.0) \resizebox*{5.0cm}{5.0cm}{\includegraphics{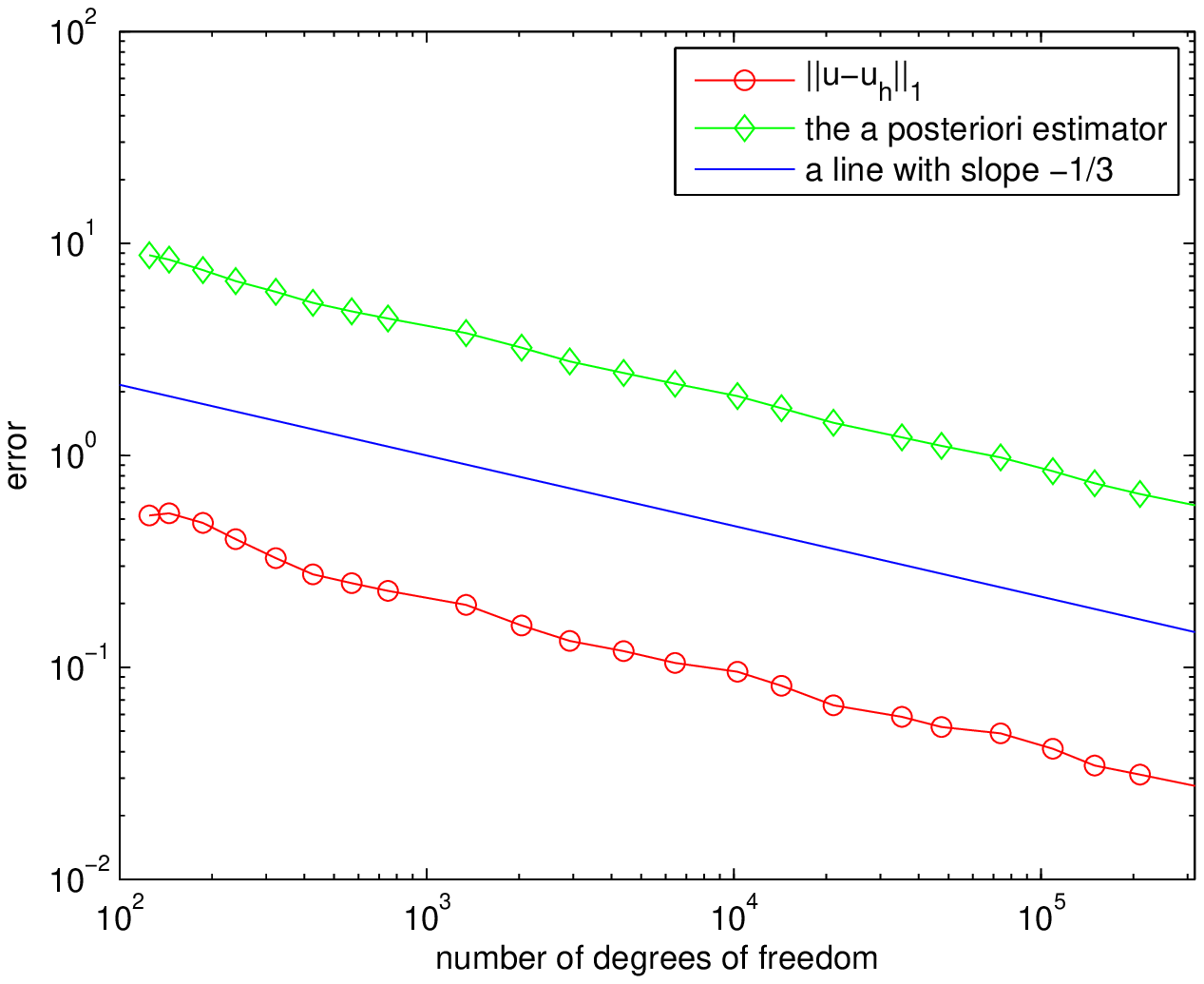}}\end{picture}
 \caption{The convergence curves
 of {\bf Example 2}
 using linear finite elements }\label{nonlinear_P1}
\end{minipage}
\hfill
\begin{minipage}[t]{5.0cm}
\begin{picture}(5.0,5.0)\resizebox*{5.0cm}{5.0cm} {\includegraphics{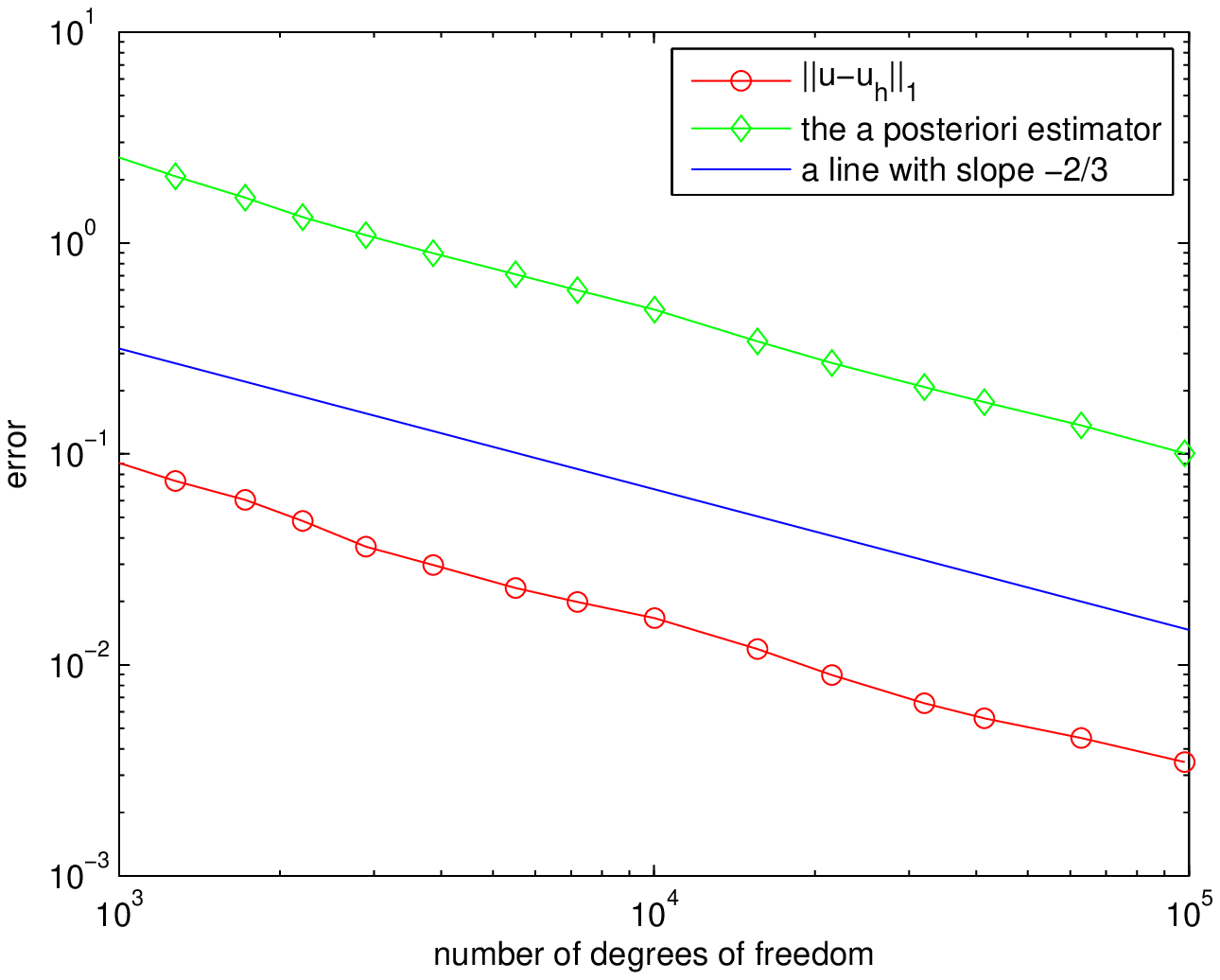}}\end{picture}
 \caption{The convergence curves
 of {\bf Example 2}
 using quadratic finite elements }\label{nonlinear_P2}
\end{minipage}
\end{center}
\end{figure}

Fig. \ref{nonlinear_P1_mesh} and Fig. \ref{nonlinear_P2_mesh} are
two adaptively refined meshes, which show that the error indicator
is good. It is shown from Fig. \ref{nonlinear_P1} and Fig.
\ref{nonlinear_P2} that $\|u - u_h\|_1$ is proportional to the a
posteriori error estimators, which implies  the a posteriori error
estimators given in section \ref{nonlinear} are efficient. Besides,
similar conclusions to that of Example 1 can be obtained from Fig.
\ref{nonlinear_P1} and Fig. \ref{nonlinear_P2}, too.

{\bf Example 3}. Consider the Kohn-Sham equation for helium atoms:
\begin{eqnarray*}\label{atom-He}
\left(-\frac {1}{2}\Delta  -\frac {2}{|x|} +
\int\frac{\rho(y)}{|x-y|}dy+ V_{xc}\right) u = \lambda u
~~\mbox{in}~ \mathbb{R}^3,
\end{eqnarray*}
and $\int_{\mathbb{R}^3}{|u|^2}=1$, here $\rho=2|u|^2.$ In our
computation of the ground state energy, we solve the following
nonlinear eigenvalue problem: Find $(\lambda,u)\in \mathbb{R}\times
H^1_0(\Omega)$ such that $\int_{\Omega}{|u|^2}dx=1$ and
\begin{equation}\label{example1}
\left\{\begin{array}{rcl} \displaystyle \left(-\frac {1}{2}\Delta -
\frac {2}{|x|}+ \int\frac{\rho(y)}{|x-y|}dy+ V_{xc}\right) u &=& \lambda u ~~\mbox{in} ~\Omega, \\
u &=& 0 ~~ \mbox{on}~\partial\Omega,
\end{array}\right.
\end{equation}
where $\Omega=(-10.0, 10.0)^3$, and $V_{xc}(\rho) =
-\frac{3}{2}\alpha(\frac{3}{\pi}\rho)^{\frac{1}{3}}$ with
$\alpha=0.77298$.  Since (\ref{example1}) is a nonlinear eigenvalue
problem, we need to linearize and solve them iteratively, which is
called the self-consistent approach
\cite{Beck-00,gong-shen-zhang-zhou-08,Kohn-Sham-65,perdew-zunger-81}.
In our computation, a Broyden-type quasi-Newton method
\cite{srivastava} were used.

In 1989, White \cite{White-89} computed  helium atoms over uniform
cubic grids and obtained ground state energy -2.8522 a.u. by using
500,000 finite element bases. While the ground state energy of
helium atoms in Software package fhi98PP \cite{Fuchs-Scheffler-99}
is -2.8346 a.u., which we take as a reference.

\begin{figure}[htbp]
\begin{center}
\setlength{\unitlength}{1cm}
\begin{minipage}[t]{5.0cm}
\begin{picture}(5.0,5.0)\resizebox*{5cm}{5cm}
{\includegraphics{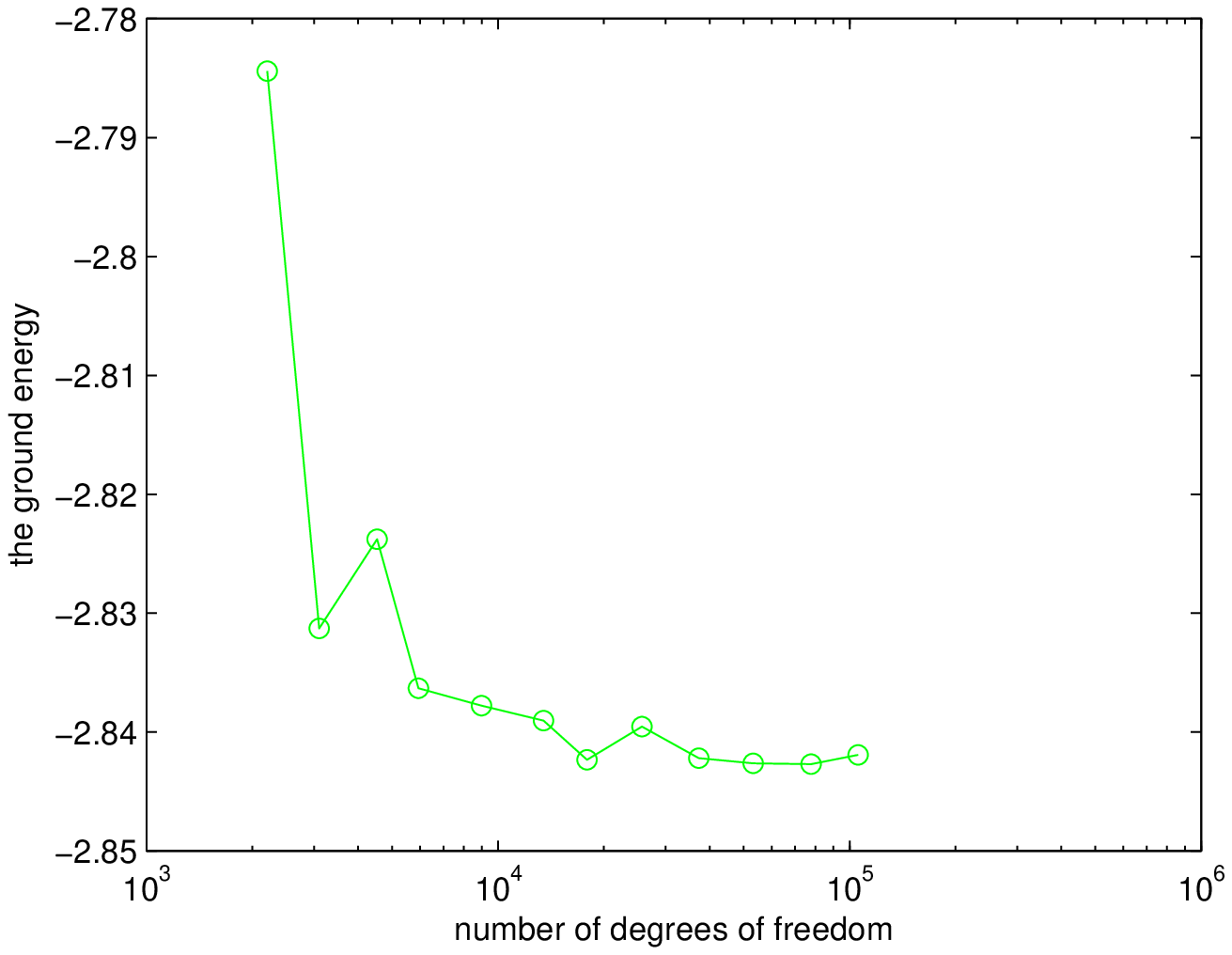}}\end{picture}
\caption{The ground state energy using linear finite elements
}\label{energy_P1}
\end{minipage}
\hfill
\begin{minipage}[t]{5.0cm}
\begin{picture}(5.0,5.0)\resizebox*{5cm}{5cm}
{\includegraphics{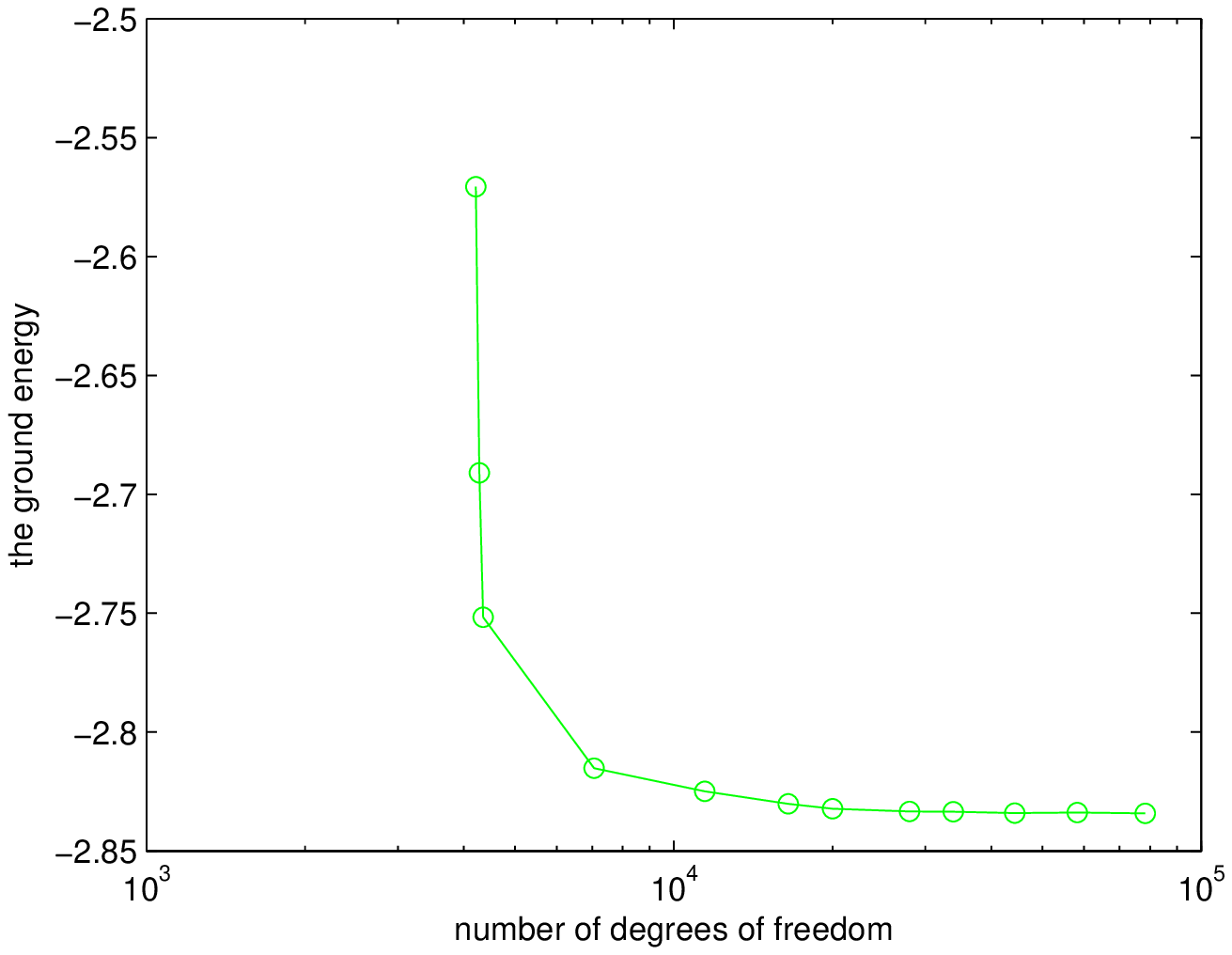}}\end{picture}
\caption{The ground state energy using quadratic finite elements}
\label{energy_P2}
\end{minipage}
\end{center}
\end{figure}

\begin{figure}[htbp]
\begin{center}
\setlength{\unitlength}{1cm}
\begin{minipage}[t]{5.0cm}
\begin{picture}(5.0,5.0)\resizebox*{5cm}{5cm}
{\includegraphics{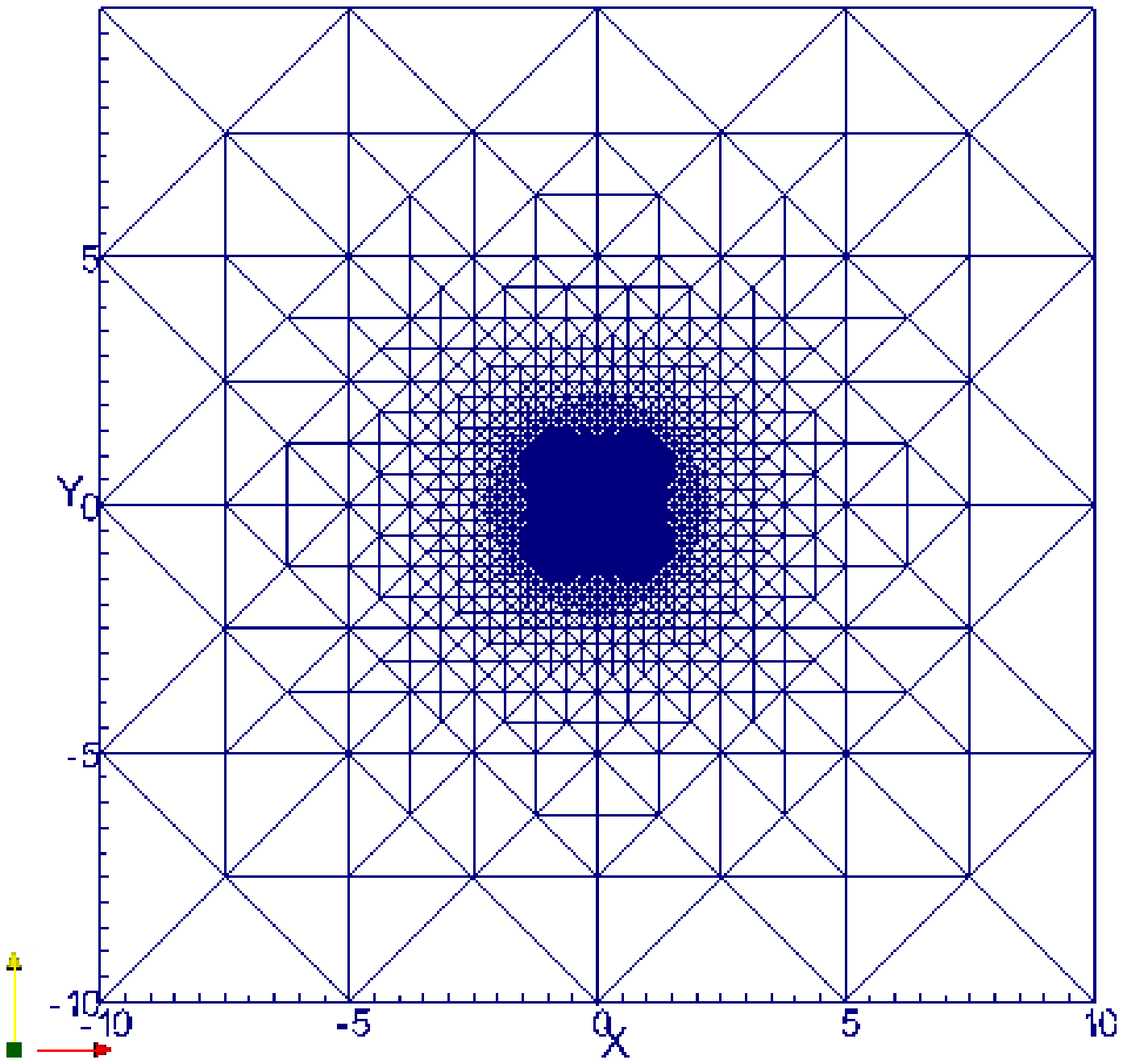}}\end{picture}
\begin{center} Z=0.0 \end{center}\caption{The cross-section of an
adaptive mesh of {\bf Example 3} using linear finite
elements}\label{he_P1_mesh}
\end{minipage}
\hfill
\begin{minipage}[t]{5.0cm}
\begin{picture}(5.0,5.0)\resizebox*{5cm}{5cm}
{\includegraphics{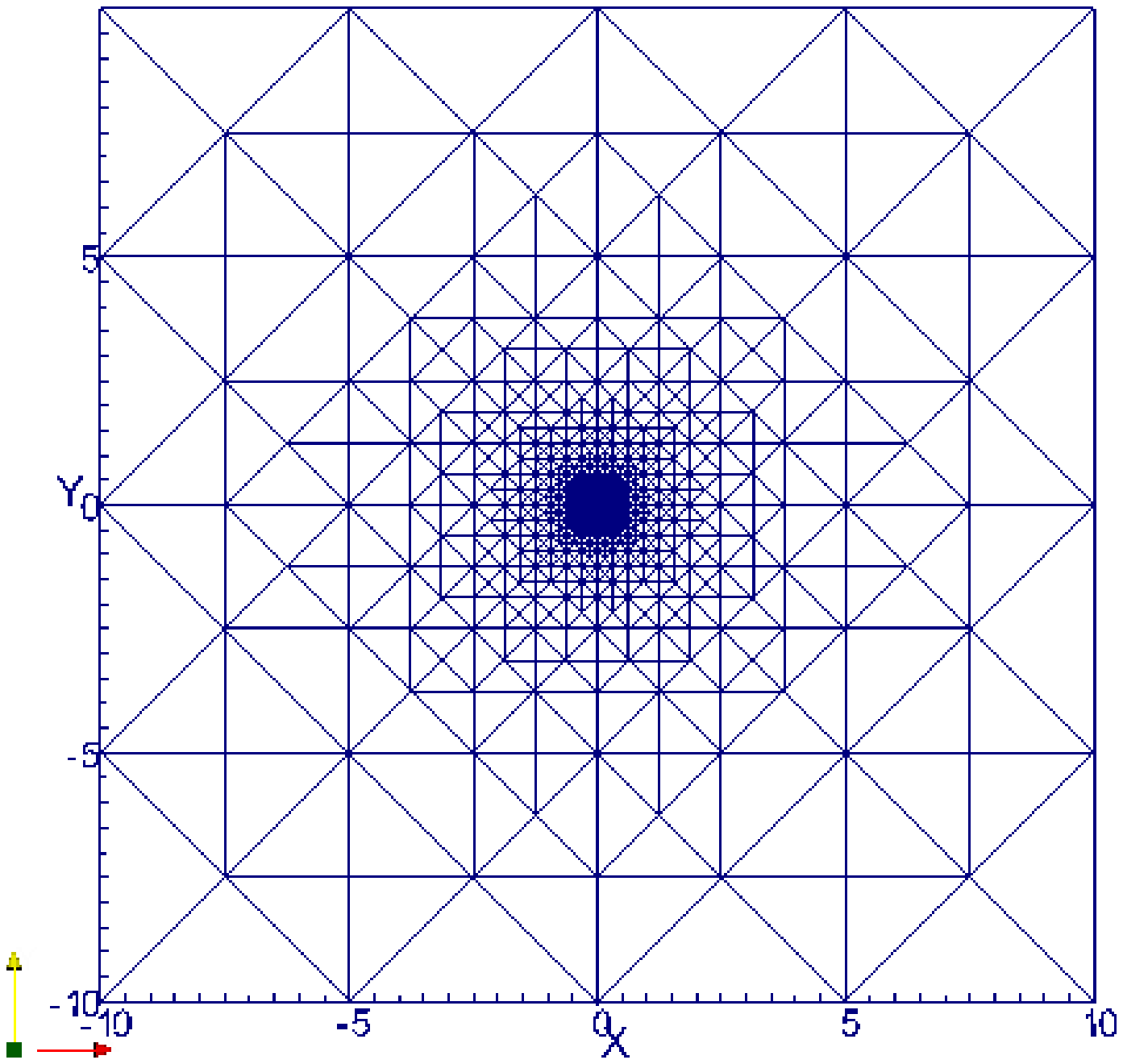}}\end{picture}
\begin{center} Z=0.0 \end{center}\caption{The cross-section of an
adaptive mesh of {\bf Example 3} using quadratic finite
elements}\label{he_P2_mesh}
\end{minipage}
\end{center}
\end{figure}

\begin{figure}[htbp]
\begin{center}
\setlength{\unitlength}{1cm}
\begin{minipage}[t]{5.0cm}
\begin{picture}(5.0,5.0)\resizebox*{5cm}{5cm}
{\includegraphics{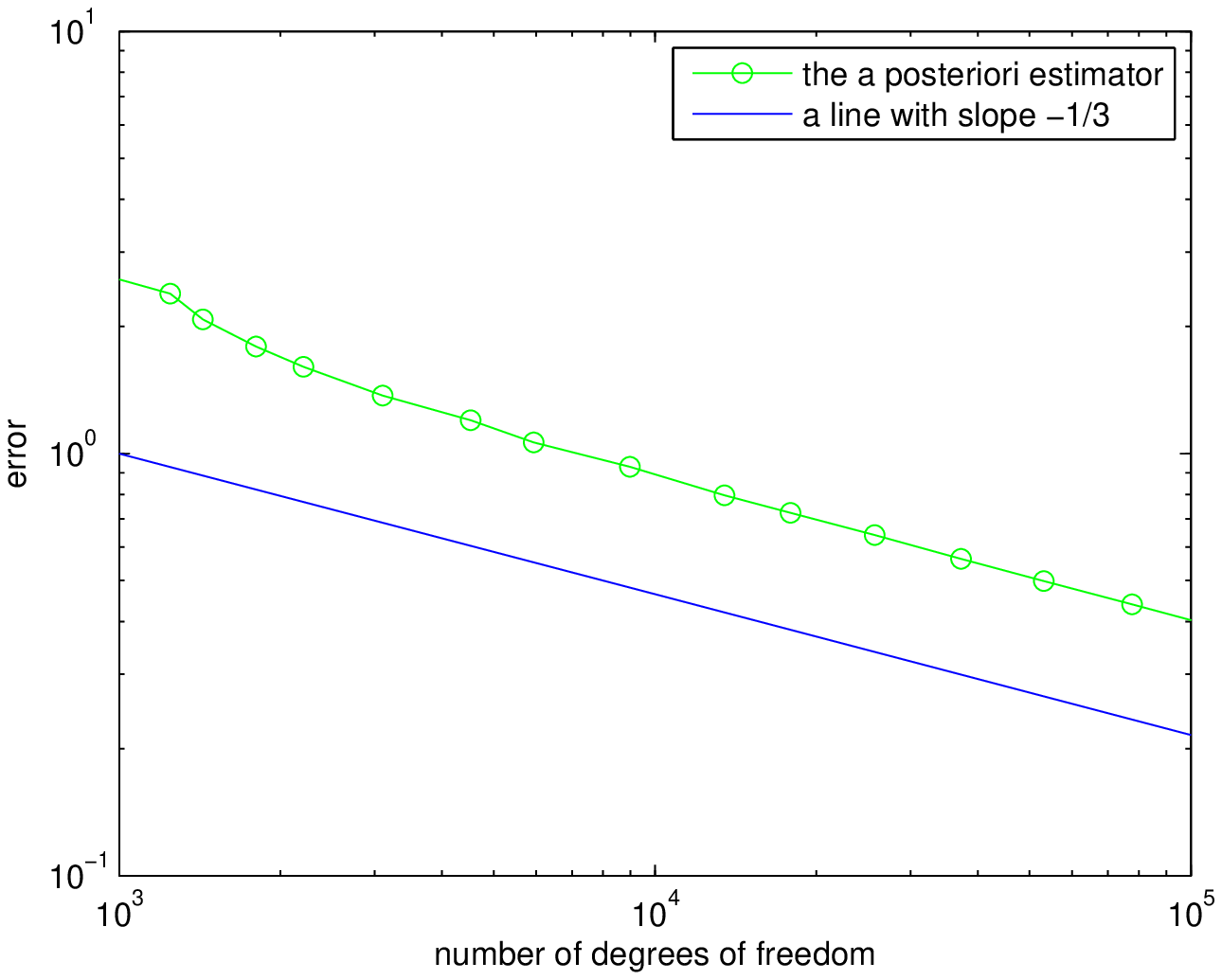}}\end{picture} \caption{The
convergence curve
 of {\bf Example 3} using linear finite elements }\label{he_P1}
\end{minipage}
\hfill
\begin{minipage}[t]{5.0cm}
\begin{picture}(5.0,5.0)\resizebox*{5cm}{5cm}
{\includegraphics{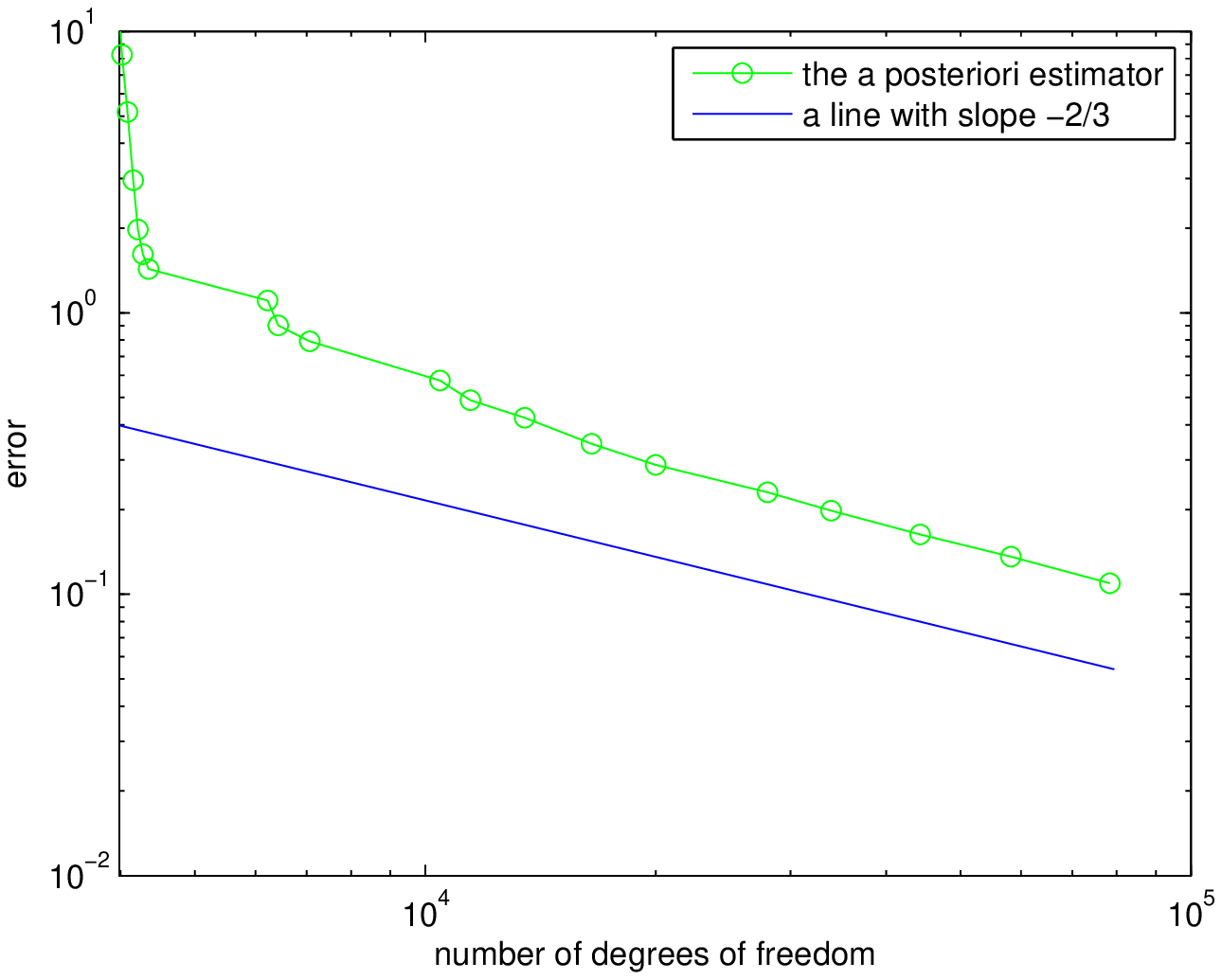}}\end{picture} \caption{The
convergence curve
 of {\bf Example 3} using quadratic finite elements} \label{he_P2}
\end{minipage}
\end{center}
\end{figure}

Our results are displayed in Fig. \ref{energy_P1}, Fig.
\ref{energy_P2}, Fig. \ref{he_P1_mesh}, Fig. \ref{he_P2_mesh}, Fig.
\ref{he_P1}, and Fig. \ref{he_P2}. It is seen from Fig.
\ref{energy_P2} that  the ground state energy in our computation is
close to the reference with less 100,000 degrees of freedom when the
quadratic finite element discretization is used. Some cross-sections
of the adaptively refined meshes are displayed in Fig.
\ref{he_P1_mesh} and Fig.\ref{he_P2_mesh}. Since we do not have the
exact solution, we list the convergence curves of the a posteriori
error estimators in Fig. \ref{he_P1} and Fig. \ref{he_P2} only. It
is shown from these figures  that the a posteriori error estimators
given in section \ref{nonsmooth} are efficient.\vskip 0.2cm



{\sc Acknowledgements.} The authors would like to thank Mr. Huajie
Chen, Dr. Xiaoying Dai, and Prof. Lihua Shen for their stimulating
discussions and fruitful cooperations that have motivated this work.

\end{document}